\documentclass{amsart}
\usepackage[utf8]{inputenc}
\usepackage[dvips]{graphicx}
\usepackage{amssymb}
\usepackage{latexsym}
\usepackage{xypic}
\usepackage{multirow}
\usepackage{multicol}

\usepackage{array}
\usepackage{epstopdf}
\usepackage{ragged2e}
\usepackage{graphicx}
\usepackage{newclude}
\usepackage{relsize}		
\usepackage{hyperref}
\usepackage{xcolor}

\usepackage{enumerate}
\usepackage{float} 			

\usepackage{verbatim}

\newtheorem{theorem}{Theorem}[section]

\theoremstyle{remark}

\newtheorem{definition}[theorem]{Definition}

\newcommand{\nin}{\not\in}

\newcommand{\set}[1]{\left\{#1\right\}}
\newcommand{\SET}[2]{\left\{#1\,\middle|\,#2\right\}}
\newcommand{\parentesis}[1]{\left(#1\right)}
\newcommand{\corchetes}[1]{\left[#1\right]}
\newcommand{\valorabs}[1]{\left|#1\right|}

\newcommand{\prodint}[1]{\left\langle #1 \right\rangle}

\newcommand{\R}{\mathbb{R}}
\newcommand{\C}{\mathbb{C}}

\newcommand{\Z}{\mathbb{Z}}

\newcommand{\Q}{\mathbb{Q}}
\newcommand{\Ss}{\mathbb{S}}

\newcommand{\id}{\mathit{id}}

\newcommand{\PSL}{\text{PSL}\parentesis{3,\C}}

\newcommand{\fix}{\text{Fix}}
\newcommand{\kernel}{\text{Ker}}
\newcommand{\core}{\text{Core}}
\newcommand{\rank}{\text{rank}}

\newcommand{\Ell}{\text{Ell}}

\newtheorem{thm}{Theorem}[section]
\newtheorem{lem}[thm]{Lemma}

\newtheorem{prop}[thm]{Proposition}
\newtheorem{defn}[thm]{Definition}

\numberwithin{equation}{section}

\title[Representations of Solvable Subgroups of $\PSL$]{\sc{Representations of Solvable Subgroups of $\PSL$}} 

\author{Mauricio Toledo-Acosta}
\address{Universidad de Sonora, Departamento de Matemáticas, Boulevard Luis Encinas y Rosales s/n Col. Centro, CP 83000, Hermosillo, Mexico.}
\email{mauricio.toledo@unison.mx}
\address{Universidad Virtual del Estado de Guanajuato, Hermenegildo Bustos 129 A Sur, Centro, C.P. 36400, Purísima del Rincón, Mexico.}
\email{getoledo@uveg.edu.mx}

\thanks{Partially supported by grants of projects PAPPIT UNAM IN101816, PAPPIT UNAM IN110219, CONACYT 282937, FORDECYT 265667} 


\begin{document}


\begin{abstract} 
In this paper, we give a complete description of the representations of all upper triangular complex Kleinian subgroups of $\PSL$. In \cite{toledo2021dynamics} we show that any solvable group is virtually triangularizable and can be constructed as the semidirect product of two layers of parabolic elements and two layers of loxodromic elements. There are five families of purely parabolic discrete groups of $\PSL$ \cite{barrera2022}, therefore, the parabolic part of any upper triangular group belongs to one of these five families. In this paper we study which loxodromic elements can be combined with the elements of the parabolic part of an upper triangular discrete subgroup of $\PSL$ in each of these five cases. These parabolic elements impose strong restrictions on the type, and number, of loxodromic elements that can be present in the group. We show that up to conjugation, there are 16 types of upper triangular complex Kleinian groups in $\PSL$ containing loxodromic elements. These results are a further step towards the completion of the study of elementary discrete subgroups of $\PSL$. 
\end{abstract}

\maketitle

\section{Introduction}

Kleinian groups are discrete subgroups of $\text{PSL}\left(2,\mathbb{C}\right)$, the group of biholomorphic automorphisms of the complex projective line $\mathbb{CP}^1$, acting properly and disconti-nuously on a non-empty region of $\mathbb{CP}^1$. Kleinian groups have been thoroughly studied since the end of the 19th century by Lazarus Fuchs, Felix Klein, Henri Poincar\'e, and many others. Kleinian groups have played a major role in several fields of mathematics, such as Riemann surfaces and Teichm\"uller theory, automorphic forms, holomorphic dynamics, conformal and hyperbolic geometry, etc. For a detailed study of Kleinian groups, see \cite{maskit}, \cite{taniguchi}. In \cite{seade2001actions}, Jos\'e Seade and Alberto Verjovsky introduced the notion of complex Kleinian groups, which are discrete subgroups of $\text{PSL}\left(n+1,\mathbb{C}\right)$ acting properly and discontinuously on an open invariant subset of the complex projective space $\mathbb{CP}^n$. In the last decade, there has been a great effort to complete the study of the dynamics of complex Kleinian groups in dimension 2 (see, for example \cite{bcn11cuatrolineas,bcn14unalinea,bcn16,bgn18,ckg_libro,cs2014}). In this regard, one of the remaining pieces is the full description of the discrete subgroups of $\text{PSL}\left(3,\mathbb{C}\right)$ which have \emph{simple} dynamics. These subgroups of $\text{PSL}\left(3,\mathbb{C}\right)$ are the analogous of elementary subgroups of $\text{PSL}\left(2,\mathbb{C}\right)$. However, there is no \emph{standard} definition of elementary subgroups of $\text{PSL}\left(3,\mathbb{C}\right)$.

In the context of Kleinian groups, \emph{elementary groups} are discrete subgroups of $\text{PSL}\left(2,\mathbb{C}\right)$ such that the limit set is a finite set, in which case the limit set has 0, 1, or 2 points. These groups have \emph{simple} dynamics. In complex dimension 2, the Kulkarni limit set is either a finite union of complex lines (1, 2 or 3) or it contains an infinite number of complex lines. On the other hand, the limit set either contains a finite number of lines in general position (1, 2, 3 or 4) or it contains infinitely many lines in general position \cite{bcn16}. Therefore, one could define elementary groups in complex dimension 2 as discrete subgroups of $\text{PSL}\left(3,\mathbb{C}\right)$ such that the Kulkarni limit set contains a finite number of lines, or discrete subgroups of $\text{PSL}\left(3,\mathbb{C}\right)$ such that the Kulkarni limit set contains a finite number of lines in general position. However, this definition would not be useful in complex dimension greater than 2 since there is no similar classification of the Kulkarni limit set in this case. 

Another alternative to characterize elementary complex Kleinian groups is solvability. This approach was first explored in \cite{toledo2021dynamics}, and it is continued in this paper. In the first part \cite{toledo2021dynamics}, we studied the dynamics of solvable subgroups of $\PSL$. We proved that solvable groups are virtually triangularizable and, up to a finite-index subgroup, can be decomposed in four layers, via the semidirect product of four types of elements.  We also showed that solvable groups act properly and discontinuously on the complement of, either, a line, two lines, a line and a point outside of the line or a pencil of lines passing through a point. 

In this second part, we provide a full description of the representations of solvable subgroups of $\PSL$, up to a finite index subgroup. We build upon the work done in \cite{barrera2022}, in this paper the authors prove that, up to conjugation, there are five families of purely parabolic discrete groups in $\PSL$ (see Theorem \ref{thm_descripcion_parte_parabolica}). These purely parabolic groups play the role of the two innermost layers of discrete upper triangular subgroups of $U_+$ (see Theorem \ref{thm_descomposicion_caso_noconmutativo}), they form a \emph{parabolic core} for the group. In this paper, we describe the loxodromic elements which can be combined with every possible \emph{parabolic core} in order to generate upper triangular discrete groups. As we will evidence, the more parameters the parabolic part of a group has, the fewer loxodromic elements can be added to the group while preserving its discreteness.

We show that up to conjugation, there are 16 types of upper triangular complex Kleinian groups in $\PSL$ containing loxodromic elements. These results are a further step towards the completion of the study of elementary discrete subgroups of $\PSL$. 

Before stating the main theorem of this paper, it is necessary to introduce some notation. We denote by $\text{PSL}\left(3,\mathbb{C}\right)$ the group of biholomorphic automorphisms of the complex projective plane $\mathbb{CP}^2$. We denote by $U_+\subset\PSL$ the upper triangular subgroups of $\PSL$. We list the following families of purely parabolic discrete subgroups of $\PSL$, 

\begin{enumerate}[(1)]
\item Torus groups,
	$$\mathcal{T}(W) = \SET{\corchetes{\begin{array}{ccc}
1 & 0 & x\\
0 & 1 & y\\
0 & 0 & 1\\
\end{array}}}{(x,y)\in W}$$
where $W\subset\C^2$ is a additive subgroup.
		 
\item Dual torus groups,
	$$\mathcal{T}^\ast(W)=\SET{\corchetes{\begin{array}{ccc}
1 & x & y\\
0 & 1 & 0\\
0 & 0 & 1\\
\end{array}}}{(x,y)\in W}$$
where $W\subset\C^2$ is a discrete additive subgroup with $r(W)\leq 2$.
\end{enumerate}	

The names given to these families, and the families defined in Theorem \ref{thm_descripcion_parte_parabolica}, were taken from \cite{barrera2022}. They stem from the fact that they are subgroups of fundamental groups of such complex surfaces. For a more detailed description of these families see \cite{barrera2022,kodaira1964structure}. 

For the sake of clarity, we divide the main theorem of this paper into four parts. These four parts describe, up to conjugation, all upper triangular, non-cyclic, discrete groups containing loxodromic elements. 

\begin{thm}\label{thm:main_1}
Let $\tilde{\Gamma}\subset U_+$ be a non-cyclic, discrete abelian group containing loxodromic elements, then there exists a finite index subgroup $\Gamma\subset\tilde{\Gamma}$ conjugate to one of the following groups.

\begin{enumerate}

\item 
    $$\SET{\text{Diag}(\alpha^n,\beta^m,1)}{n,m\in\Z},$$
    where $\alpha,\beta\in \C^\ast$, $|\alpha|\neq 1$ or $|\beta|\neq 1$.
    
\item 
    $$\SET{\corchetes{\begin{array}{ccc}
    \mu(w) & w\mu(w) & 0\\
    0 & \mu(w) & 0\\
    0 & 0 & \mu(w)^{-2}\\
    \end{array}}}{w\in W}$$
    where $W\subset\C$ is a discrete additive group such that $r(W)\leq 3$ and $\mu:(W,+)\rightarrow(\C^\ast,\cdot)$ is a group morphism satisfying that for every sequence of distinct elements $\set{w_k}\subset W$ it holds that $\mu(w_k)\rightarrow 0$ or $\mu(w_k)\rightarrow \infty$.
    
\end{enumerate}

\end{thm}

\begin{thm}\label{thm:main_2}
Let $\tilde{\Gamma}\subset U_+$ be a non-cyclic, discrete non-abelian group containing loxodromic elements. If the parabolic part of $\tilde{\Gamma}$ is conjugate to a torus group, then there exists a finite index subgroup $\Gamma\subset\tilde{\Gamma}$ conjugate to the group

    $$\mathcal{T}(W)
    \rtimes
    \SET{
    \corchetes{
    \begin{array}{ccc}
    \alpha^{2n}\beta^n & 0 & 0\\
    0 & \alpha^n\beta^{2n} & 0\\
    0 & 0 & 1\\
    \end{array}
    }
    }{n\in\Z},
    $$
    where $r(W)=3$ and $\alpha,\beta\in\C^\ast$, $\alpha\neq 1$, $\alpha\beta^2\nin\R$.

\end{thm}

\begin{thm}\label{thm:main_3}
Let $\tilde{\Gamma}\subset U_+$ be a non-cyclic, discrete non-abelian group containing loxodromic elements. If the parabolic part of $\tilde{\Gamma}$ is conjugate to a dual torus group $\mathcal{T}^\ast(W)$ such that $r(W)=2$, then there exists a finite index subgroup $\Gamma\subset\tilde{\Gamma}$ conjugate to one of the following groups.

\begin{enumerate}

\item 
    $$\mathcal{T}^\ast(W)\rtimes 
    \prodint{
    \corchetes{
    \begin{array}{ccc}
    p & \gamma_{12} & \gamma_{13}\\
	0 & 1 & \frac{q}{p}\\
	0 & 0 & 1\\
	\end{array}}
    },$$
    where $W\cong \prodint{(1,0),(0,1)}$, $p\in\Z\setminus\set{-1,0,1}$, $q\in\Z\setminus\set{0}$ and $\gamma_{12},\gamma_{13}\in\C$.

\item 
    $$\mathcal{T}^\ast(W)\rtimes 
    \prodint{
    \corchetes{
    \begin{array}{ccc}
    p_1 & \gamma_{12} & \gamma_{13}\\
	0 & 1 & \frac{q_1}{p_1}\\
	0 & 0 & 1\\
	\end{array}},
	\corchetes{\begin{array}{ccc}
	p_2 & \mu_{12} & \mu_{13}\\
	0 & 1 & \frac{q_2}{p_2}\\
	0 & 0 & 1\\
	\end{array}}
    },$$
    where $W\cong \prodint{(1,0),(0,1)}$, $p_1,p_2\in\Z\setminus\set{-1,0,1}$, $q_1,q_2\in\Z\setminus\set{0}$ and $\gamma_{12},\gamma_{13},\mu_{12},\mu_{13}\in\C$ satisfying the conditions of Lemma \ref{lem_kmn4_combinaciones_Gamma1}(2).
    
\item 
    $$\mathcal{T}^\ast(W)\rtimes 
    \prodint{\gamma:=
     \corchetes{\begin{array}{ccc}
	pq & \gamma_{12} & \gamma_{13}\\
	0 & q & r\\
	0 & 0 & p\\
	\end{array}}
    },$$
    where $W\cong \prodint{(1,0),(0,1)}$, $p,q\in\Z\setminus\set{0}$ such that $\gamma$ is loxodromic, $r\in\Z$ and $\gamma_{12},\gamma_{13}\in\C$.
    
\item 
    $$\mathcal{T}^\ast(W)\rtimes 
    \prodint{
    \corchetes{\begin{array}{ccc}
	p+q\text{Re}(y)+iq\text{Im}(y) & \gamma_{12} & \gamma_{13}\\
	0 & 1 & \gamma_{23}\\
	0 & 0 & 1\\
	\end{array}}
    },$$
    $W\cong \prodint{(0,1),(0,y)}$, with $y\nin\R$. Also, $p,q\in\Z$ such that $\valorabs{p}+\valorabs{q}\neq 0$ and $\valorabs{\parentesis{p+q\text{Re}(y)}+iq\text{Im}(y)}\neq 1$, $\gamma_{12},\gamma_{23}\in\C$ and $\gamma_{23}\in\C^\ast$.
    
\item 
    $$\mathcal{T}^\ast(W)\rtimes 
    \prodint{\gamma:=
    \corchetes{\begin{array}{ccc}
	\alpha & \gamma_{12} & \gamma_{13}\\
	0 & \alpha^{-2}z_{p,q} & \gamma_{23}\\
	0 & 0 & \alpha z_{p,q}^{-1}\\
	\end{array}}
    },$$
    $W\cong \prodint{(0,1),(0,y)}$, with $y\nin\R$. Also, $\beta\in\C^{\ast}$, $z_{p,q}=\parentesis{p+q\text{Re}(y)}+iq\text{Im}(y)$, $\valorabs{p}+\valorabs{q}\neq 0$, $\gamma_{12},\gamma_{23},\gamma_{23}\in\C$. Finally, $\gamma$ must be loxodromic. 
    
\item 
    $$\mathcal{T}^\ast(W)\rtimes 
    \prodint{\gamma:=\corchetes{\begin{array}{ccc}
	p+qx & \gamma_{12} & \gamma_{13}\\
	0 & 1 & 0\\
	0 & 0 & (p+qx)^{-1}\beta^{-3}\\
	\end{array}}
    },$$
    $W\cong \prodint{(1,0),(x,0)}$, with $x\nin\R$. Also, $p,q\in\Z$ satisfying the conditions of Lemma \ref{lem_kmn4_combinaciones_Gamma3}(ii), $\beta\in\C^\ast$, $\gamma_{12}, \gamma_{13}\in\C$, and $\gamma$ is loxodromic.
    
\item 
    $$\mathcal{T}^\ast(W)\rtimes 
    \prodint{\gamma,
    \mu:=\corchetes{\begin{array}{ccc}
	p'+q'x & \mu_{12} & 0\\
	0 & 1 & 0\\
	0 & 0 & (p'+q'x)^{-1}\beta_2^{-3}\\
	\end{array}}
    },$$
    $W\cong \prodint{(1,0),(x,0)}$, with $x\nin\R$, and $\gamma$ defined by (6). Also, $p',q'\in\Z$ satisfying the conditions of Lemma \ref{lem_kmn4_combinaciones_Gamma3}(iii), and $\mu$ is loxodromic.
\end{enumerate}
\end{thm}

\begin{thm}\label{thm:main_4}
Let $\tilde{\Gamma}\subset U_+$ be a non-cyclic, discrete non-abelian group containing loxodromic elements. If the parabolic part of $\tilde{\Gamma}$ is conjugate to a dual torus group $\mathcal{T}^\ast(W)$ such that $r(W)=1$, then there exists a finite index subgroup $\Gamma\subset\tilde{\Gamma}$ conjugate to one of the following groups.

\begin{enumerate}

\item 
    $$\mathcal{T}^\ast(W)\rtimes 
    \prodint{
    \gamma :=
    \corchetes{\begin{array}{ccc}
	p\alpha & \gamma_{12} & \gamma_{13}\\
	0 & \alpha & 0\\
	0 & 0 & p^{-1}\alpha^{-2}\\
	\end{array}}
    },$$
    where $W\cong \prodint{(1,0)}$, $p\in\Z\setminus\set{0,1}$, $\alpha\in\C^{\ast}$, $\valorabs{\alpha}\neq 1$, and $\gamma_{12},\gamma_{13}\in\C$.

\item 
    $$\mathcal{T}^\ast(W)\rtimes 
    \prodint{
    \gamma,
    \mu:=\corchetes{\begin{array}{ccc}
	q\beta & \frac{\beta jpq}{1-p} & \mu_{13}\\
	0 & \beta & 0\\
	0 & 0 &q^{-1}\beta^{-2}\\
	\end{array}}
    },$$
    where $W\cong \prodint{(1,0)}$, $\gamma$ defined by (1). Also, $\beta\in\C^\ast$, $j,q\in\Z$, $q\neq 0$ such that $\mu$ is loxodromic, and $\mu_{13}\in\C$, if $p^2\alpha^3\neq 1$, $\mu_{13}=0$. 
    
\item 
    $$\mathcal{T}^\ast(W)\rtimes 
    \prodint{
    \gamma,
    \mu,\eta:=\corchetes{\begin{array}{ccc}
	r\delta & \frac{\delta k pr}{1-p} & \eta_{13}\\
	0 & \delta & 0\\
	0 & 0 &q^{-1}\delta^{-2}\\
	\end{array}}
    },$$
    where $W\cong \prodint{(1,0)}$ and $\gamma,\mu$ are given by (1) and (2). Also, $\delta\in\C^\ast$, $r\in\Z\setminus\set{-1,0,1}$, $r\delta^3\neq 1$ such that $\eta$ is loxodromic, and $\eta_{13},k\in\C$ satisfying the conditions of Lemma \ref{lem_kmn4_combinaciones_Gamma4}. 
    
\item 
    $$\mathcal{T}^\ast(W)\rtimes 
    \prodint{
    \corchetes{\begin{array}{ccc}
	p & \gamma_{12} & \gamma_{13}\\
	0 & 1 & \gamma_{23}\\
	0 & 0 & 1\\
	\end{array}}
    },$$
	where $W\cong \prodint{(0,1)}$, $p\in\Z\setminus\set{-1,0,1}$, $\gamma_{12},\gamma_{13}\in\C$, $\gamma_{23}\in\C^\ast$.

\item 
    $$\mathcal{T}^\ast(W)\rtimes 
    \prodint{
    \gamma:=\corchetes{\begin{array}{ccc}
	\alpha & \gamma_{12} & \gamma_{13}\\
	0 & p\alpha^{-2} & \gamma_{23}\\
	0 & 0 & p^{-1}\alpha\\
	\end{array}}
    },$$
	where $W\cong \prodint{(0,1)}$, $\alpha\in\C^{\ast}$, $p\in\Z\setminus\set{0}$, $p^2\neq \alpha^3$ and $\gamma$ is loxodromic.

\item 
    $$\mathcal{T}^\ast(W)\rtimes 
    \prodint{
    \gamma,
    \mu=\corchetes{\begin{array}{ccc}
	\beta & 0 & \frac{jp}{1-p}\beta\\
	0 & q\beta^{-2} & 0\\
	0 & 0 & q^{-1}\beta\\
	\end{array}}
    },$$
	where $W\cong \prodint{(0,1)}$ and $\gamma$ is given by (5). Also, $\beta\in\C^\ast$, $q,j\in\Z\setminus\set{0}$, $q^2\neq\beta^3$ and $\mu$ is loxodromic.
\end{enumerate}
\end{thm}

The paper is organized as follows: In Section \ref{sec_background} we present a brief background necessary for the next sections, we also summarize the necessary results from \cite{toledo2021dynamics}. In Section \ref{sec_representation}, we state and prove the technical results describing or dismissing combinations of parabolic and loxodromic elements. These results together provide a full description of the representation of upper triangular discrete triangular subgroups of $\PSL$. Finally, in Section \ref{sec:proofs}, we prove the main theorems.

\section{Notation and Background} \label{sec_background}

\subsection{Complex Kleinian groups}

The complex projective plane $\mathbb{CP}^2$ is defined as
	$$\mathbb{CP}^2=(\mathbb{C}^{3}\setminus\{0\})/\mathbb{C}^\ast,$$
where $\mathbb{C}^\ast:=\mathbb{C}\setminus\{0\}$ acts by the usual scalar multiplication. Let $\left[\;\right]:\mathbb{C}^3\setminus\{0\}\rightarrow\mathbb{CP}^2$ be the quotient map. We denote the projectivization of the point $x=(x_1,x_2,x_3)\in\mathbb{C}^3$ by $[x]=\left[x_1:x_2:x_3\right]$. We denote by $e_1,e_2,e_3$ the projectivization of the canonical base of $\mathbb{C}^3$.

Let $\mathcal{M}_{3}\left(\mathbb{C}\right)$ be the group of all square matrices of size $3\times 3$ with complex coefficients, and let $\text{SL}\left(3,\mathbb{C}\right)\subset\mathcal{M}_{3}\left(\mathbb{C}\right)$ (resp. $\text{GL}\left(3,\mathbb{C}\right)$) be the subgroup of matrices with determinant equal to $1$ (resp. not equal to $0$). The group of biholomorphic automorphisms of $\mathbb{CP}^2$ is given by 
	$$\text{PSL}\left(3,\mathbb{C}\right)\cong\text{PGL}\left(3,\mathbb{C}\right):=\text{GL}\left(3,\mathbb{C}\right)/\{\text{scalar matrices}\}.$$
If $g\in\text{PSL}\left(3,\mathbb{C}\right)$ (resp. $z\in\mathbb{CP}^2$), we denote by $\mathbf{g}\in\text{SL}\left(3,\mathbb{C}\right)$ any of its lifts (resp. $\mathbf{z}\in\mathbb{C}^3$). We denote again by $\left[\;\right]:\mathbb{C}^3\setminus\{0\}\rightarrow\mathbb{CP}^2$ the quotient map. We denote by $\text{Fix}(g)\subset\mathbb{CP}^2$ the set of fixed points of an automorphism $g\in\text{PSL}\left(3,\mathbb{C}\right)$. We will denote by $\text{Diag}(a_1,a_2,a_3)$ to the diagonal element of $\PSL$ with diagonal entries $a_1,a_2,a_3$.

As in the case of automorphisms of $\mathbb{CP}^1$, we classify the elements of $\text{PSL}\left(3,\mathbb{C}\right)$ in three classes: elliptic, parabolic and loxodromic. However, unlike the classical case, there are several subclasses in each case (see Section 4.2 of \cite{ckg_libro}). We now give a quick summary of the subclasses of elements we will be using.

\begin{defn}
An element $g\in\text{PSL}\left(3,\mathbb{C}\right)$ is said to be:
	\begin{itemize}
	\item \emph{Elliptic} if it has a diagonalizable lift in $\text{SL}\left(3,\mathbb{C}\right)$ such that every eigenvalue has norm 1.
	\item \emph{Parabolic} if it has a non-diagonalizable lift in $\text{SL}\left(3,\mathbb{C}\right)$ such that every eigenvalue has norm 1.
	\item \emph{Loxodromic} if it has a lift in $\text{SL}\left(3,\mathbb{C}\right)$ with an eigenvalue of norm distinct of 1. Furthermore, we say that $g$ is:
		\begin{itemize}
		\item \emph{Loxo-parabolic} if it is conjugate to an element $h\in\text{PSL}\left(3,\mathbb{C}\right)$ such that
			$$\mathbf{h}=\left(
			\begin{array}{ccc}
 			\lambda & 1 & 0 \\
 			0 & \lambda & 0 \\
 			0 & 0 & \lambda^{-2}
			\end{array}\right),\;|\lambda|\neq 1.$$
		\item A \emph{complex homothety} if it is conjugate to an element $h\in\text{PSL}\left(3,\mathbb{C}\right)$ such that $\mathbf{h}=\text{Diag}\left(\lambda,\lambda,\lambda^{-2}\right)$, with $|\lambda|\neq 1$.
		\item A \emph{rational (resp. irrational) screw} if it is conjugate to an element $h\in\text{PSL}\left(3,\mathbb{C}\right)$ such that $\mathbf{h}=\text{Diag}\left(\lambda_1,\lambda_2,\lambda_3\right)$, with $\left| \lambda_1 \right|=\left|\lambda_2 \right| \neq \left|\lambda_3\right|$ and $\lambda_1\lambda_2^{-1}=e^{2\pi i \theta}$ with $\theta\in \mathbb{Q}$ (resp. $\theta\in\mathbb{R}\setminus\mathbb{Q}$).
		\item \emph{Strongly loxodromic} if it is conjugate to an element $h\in\text{PSL}\left(3,\mathbb{C}\right)$ such that $\mathbf{h}=\text{Diag}\left(\lambda_1,\lambda_2,\lambda_3\right)$, where the elements $\{|\lambda_1|,|\lambda_2|,|\lambda_3|\}$ are pairwise different.
		\end{itemize}		 
	\end{itemize}
\end{defn}

If $\Gamma\subset\PSL$ is a group, we denote by $\Gamma_p\subset\Gamma$ the subgroup generated by all parabolic elements of $\Gamma$. We will say that $\Gamma_p$ is the parabolic part of $\Gamma$.

An element $\gamma\in\Gamma$ is called a \emph{torsion element} if it has finite order. The group $\Gamma$ is a \emph{torsion free group} if the only torsion element in $\Gamma$ is the identity. We denote the rank of a group $\Gamma$ by $r(\Gamma)$.


\begin{definition}
Let $G$ be a group. The derived series $\set{G^{(i)}}$ of $G$ is defined inductively as
	$$\begin{array}{cc}
	G^{(0)}=G, & G^{(i+1)}=\left[G^{(i)},G^{(i)}\right].
	\end{array}$$
One says that $G$ is \emph{solvable} if, for some $n\geq 0$, we have $G^{(n)}=\set{\id}$. 
\end{definition}



\subsection{Upper triangular groups}

We denote the upper triangular subgroup of $\text{PSL}\left(3,\mathbb{C}\right)$ by
	$$U_{+}=\left\{\left[\begin{array}{ccc}
	a_{11} & a_{12} & a_{13} \\
	0 & a_{22} & a_{23} \\
	0 & 0 & a_{33}
	\end{array}\right]\,\middle|\,a_{11}a_{22}a_{33}=1,\text{ }a_{ij}\in\mathbb{C}\right\}.$$
Let $\lambda_{12},\lambda_{23},\lambda_{13}:(U_{+},\cdot)\rightarrow(\mathbb{C}^{\ast},\cdot)$ be the group morphisms given by
	$$\begin{array}{ccc}
	\lambda_{12}\left(\left[\alpha_{ij}\right]\right)=\alpha_{11}\alpha^{-1}_{22}, &
	\lambda_{23}\left(\left[\alpha_{ij}\right]\right)=\alpha_{22}\alpha^{-1}_{33}, &
	\lambda_{13}\left(\left[\alpha_{ij}\right]\right)=\alpha_{11}\alpha^{-1}_{33}.	
	\end{array}$$

If $\Gamma\subset U_{+}$ is a group, we simplify the notation by writing $\text{Ker}\left(\lambda_{ij}\right)$ instead of $\text{Ker}\left(\lambda_{ij}\right)\cap \Gamma$.

The following theorem states that, up to finite index, we can consider upper triangular subgroups of $\PSL$ instead of solvable subgroups of $\PSL$.

\begin{thm}[Theorem 3.6 of \cite{wehrfritz2012infinite}]\label{teo_solvable_triangularizables}
Let $G\subset\text{GL}\left(3,\mathbb{C}\right)$ be a discrete solvable subgroup, then $G$ is virtually triangularizable. 
\end{thm}

The following proposition can be found in \cite{barrera2022}.

\begin{prop}[Lemma 4.10 of \cite{barrera2022}]
\label{prop:assume_torsion_free}
Let $\Gamma\subset U_+$ be a discrete group, then $\Gamma$ contains a finite index subgroup $\Gamma_0\subset \Gamma$ such that the groups $\Gamma_0$, $\lambda_{12}(\Gamma_0)$, $\lambda_{23}(\Gamma_0)$ are torsion free and finitely generated.
\end{prop}

Now, we recall some results from \cite{toledo2021dynamics} that will be used in Section \ref{sec_representation}.

\begin{prop}[Corollary 4.8 of \cite{toledo2021dynamics}]\label{cor_no_hay_CH_en_no_conmutativos}
Let $\Gamma\subset U_+$ be a non-abelian, torsion-free discrete subgroup, then $\Gamma$ cannot contain complex homotheties.  
\end{prop}

Recall the definition of the following purely parabolic subgroup of $\Gamma$ which determines the dynamics of $\Gamma$.
	$$\core(\Gamma)=\kernel(\Gamma)\cap\kernel(\lambda_{12})\cap\kernel(\lambda_{23}).$$
We denote the elements of $\core(\Gamma)$ by
	$$g_{x,y}=\corchetes{\begin{array}{ccc}
	1 & x & y \\
	0 & 1 & 0 \\
	0 & 0 & 1\\
	\end{array}}.$$
	
	

For the sake of simplicity, we also denote
$$h_{x,y}=\corchetes{\begin{array}{ccc}
	1 & 0 & x\\
	0 & 1 & y\\
	0 & 0 & 1\\
	\end{array}},\; x,y\in\C.$$


\begin{thm}[Theorem 5 of \cite{toledo2021dynamics}]\label{thm_descomposicion_caso_noconmutativo}
Let $\Gamma\subset U_{+}$ be a non-abelian, torsion-free, complex Kleinian group, then $\Gamma$ can be written in the following way
	$$\Gamma = \text{Core}(\Gamma)\rtimes\langle\xi_1\rangle\rtimes...\rtimes\langle\xi_p\rangle\rtimes\langle\eta_1\rangle\rtimes...\rtimes\langle\eta_m\rangle\rtimes\langle\gamma_1\rangle\rtimes...\rtimes\langle\gamma_n\rangle$$
	where
	$$\begin{array}{rl}
	\lambda_{23}(\Gamma)=\langle\lambda_{23}(\gamma_1),...,\lambda_{23}(\gamma_n)\rangle, & n=r\left(\lambda_{23}(\Gamma)\right).\\
	\lambda_{12}\left(\text{Ker}(\lambda_{23})\right)=\langle\lambda_{12}(\eta_1),...,\lambda_{12}(\eta_m)\rangle, & m=r\left(\lambda_{12}(\text{Ker}(\lambda_{23})\right).\\
	\Pi\left(\text{Ker}(\lambda_{12})\cap\text{Ker}(\lambda_{23})\right)=\langle\Pi(\xi_1),...,\Pi(\xi_p)\rangle, & p=r\left(\Pi\left(\text{Ker}(\lambda_{12})\cap\text{Ker}(\lambda_{23})\right)\right).
	\end{array}$$
\end{thm}

\begin{thm}[Theorem 7 of \cite{toledo2021dynamics}]\label{thm_descomposicion_caso_noconmutativo2}
Let $\Gamma\subset U_{+}$ be a non-abelian, torsion free, complex Kleinian group, then $r(\Gamma)\leq 4$. Using the notation of Theorem \ref{thm_descomposicion_caso_noconmutativo}, it holds $k+p+m+n\leq 4$, where $k=r\parentesis{\core(\Gamma)}$.
\end{thm}

Theorem \ref{thm_descomposicion_caso_noconmutativo} gives a decomposition of the group $\Gamma$ in four layers, the first two layers are made of parabolic elements and the last two layers are made of loxodromic elements. The description of these four layers are summarized in Table \ref{fig_noconmutativo_capas}.

\begin{table}[H]
\begin{center}
  \begin{tabular}{  c  c  c  c  }
    \multicolumn{2}{c}{Parabolic} & \multicolumn{2}{c}{Loxodromic}\\
    \multicolumn{2}{c}{$\overbrace{\hspace{25ex}}$} & \multicolumn{2}{c}{$\overbrace{\hspace{45ex}}$}\\
    $\corchetes{\begin{array}{ccc}
	1 & x & y\\
	0 & 1 & 0 \\
	0 & 0 & 1\\
	\end{array}}$ & $\corchetes{\begin{array}{ccc}
	1 & x & y\\
	0 & 1 & z\\
	0 & 0 & 1\\
	\end{array}}$ & $\corchetes{\begin{array}{ccc}
	\alpha & x & y\\
	0 & \beta & z\\
	0 & 0 & \beta\\
	\end{array}}$ & $\corchetes{\begin{array}{ccc}
	\alpha & x & y\\
	0 & \beta & z\\
	0 & 0 & \gamma\\
	\end{array}}$\\ 
     & $z\neq 0$ & $\alpha\neq\beta$, $z\neq 0$ & $\beta\neq\gamma$     \\
     & & \begin{tabular}{c}
     \small{Loxo-parabolic} \\
     \end{tabular} & \begin{tabular}{c}
     \small{Loxo-parabolic} \\
     \small{Complex homothety} \\
     \small{Strongly loxodromic} \\
     \end{tabular} \\  
    $\core(\Gamma)$ & $A\setminus\kernel(\Gamma)$ & $\text{Ker}(\lambda_{23})\setminus A$ & \\
    & & & \\
    \multicolumn{2}{c}{Parabolic Part $\Gamma_p$} & Third Layer & Fourth Layer \\
    & & & 
  \end{tabular}
\end{center}
\caption{\small The decomposition of a non-abelian subgroup of $U_+$ in four layers.}
\label{fig_noconmutativo_capas}
\end{table}

\section{Representations of upper triangular groups}
\label{sec_representation}

In this section, we prove the technical lemmas necessary to give an explicit description of all upper triangular groups. These results will be organized and summarized in Section \ref{sec:proofs}. For further details, see \cite{mitesis}. 

The following result from \cite{barrera2022} describes the parabolic part of a complex Kleinian group $\Gamma\subset U_+$. Our results will describe the loxodromic elements compa-tible with these parabolic parts.


\begin{thm}[Theorem 3.1 of \cite{barrera2022}]
\label{thm_descripcion_parte_parabolica}
Let $\tilde{\Gamma}\subset\PSL$ be a complex Kleinian group without loxodromic elements, then there exists a finite index subgroup $\Gamma\subset\tilde{\Gamma}$ conjugate to one of the following groups.
	\begin{enumerate}[(1)]
	\item An elliptic group,
		$$\Ell(W,\mu)=\SET{\corchetes{\begin{array}{ccc}
	\mu(w) & w\mu(w) & 0\\
	0 & \mu(w) & 0\\
	0 & 0 & \mu(w)^{-2}\\
	\end{array}}}{w\in W}$$
	where $W\subset\C$ is a discrete subgroup and $\mu:(W,+)\rightarrow(\Ss^1,\cdot)$ is a group morphism. 
	\item A torus group $\mathcal{T}(W)$, 
	where $W\subset\C^2$ is a additive subgroup.
	\item An abelian Kodaira group,
		$$\text{Kod}_0(W,R,L)=\SET{\corchetes{\begin{array}{ccc}
	1 & x & L(x)+\frac{x^2}{2}+w\\
	0 & 1 & x\\
	0 & 0 & 1\\
	\end{array}}}{x\in R, w\in W}$$
	where $W\subset\C$ is an additive discrete subgroup, $R\subset\C$ is an additive subgroup and $L:R\rightarrow\C$ is an additive function such that
		\begin{itemize}
		\item If $R$ is discrete, then $r(W)+r(R)\leq 4$.
		\item If $R$ is not discrete, then $r(W)\leq 1$, $r(W)+r(R)\leq 4$ and
			$$\lim_{n\rightarrow\infty} L(x_n) +w_n=\infty$$
			for any sequence $\set{w_n}\subset W$, and any sequence $\set{x_n}\subset R$ such that $x_n\rightarrow 0$. 		
		\end{itemize}		 
	\item A dual torus group, where $W\subset\C^2$ is a discrete additive subgroup with $r(W)\leq 2$.

	\item An Inoue group,
		$$\text{Ino}(\mathbf{w})=\SET{\corchetes{\begin{array}{ccc}
	1 & 1 & 0\\
	0 & 1 & 0\\
	0 & 0 & 1\\
	\end{array}}^k \corchetes{\begin{array}{ccc}
	1 & \frac{p}{q} & \frac{1}{s}\\
	0 & 1 & 0\\
	0 & 0 & 1\\
	\end{array}}^l \corchetes{\begin{array}{ccc}
	1 & x & y\\
	0 & 1 & 1\\
	0 & 0 & 1\\
	\end{array}}^m}{k,l,m\in\Z},$$
	where $\mathbf{w}=(x,y,p,q,s)$; $x,y\in\C$, $p,q,s\in\Z$ such that $q,s\neq 0$, $p,q$ are co-primes and $q^2$ divides $s$. 
	\item An extended Kodaira groups of type 1 or 2,
		$$\text{Kod}_i(W,x,y,z) = \SET{
		\corchetes{\begin{array}{ccc}
    	1 & 0 & w\\
    	0 & 1 & 0\\
    	0 & 0 & 1\\
    	\end{array}} 
		\corchetes{\begin{array}{ccc}
    	1 & 1 & 0\\
    	0 & 1 & 1\\
    	0 & 0 & 1\\
    	\end{array}}^n
		\corchetes{\begin{array}{ccc}
    	1 & x+z & y\\
    	0 & 1 & z\\
    	0 & 0 & 1\\
    	\end{array}}^m}{
	\begin{array}{l}
	m,n\in\Z \\
	w\in W
	\end{array}
	},$$
	with $i=1,2$. $W\subset\C$ is a non-trivial, additive discrete group, and $x,y,z\in\C$ are such that $x\in W\setminus\set{0}$ and, either:
		\begin{enumerate}
		\item[{\footnotesize\textbf{Type $i=$1:}}] $z\nin\R$.
		\item[{\footnotesize\textbf{Type $i=$2:}}] $z\in\R\setminus\Q$ and $r(W)=1$.
		\end{enumerate}
	\end{enumerate}

\end{thm}

Given a non-abelian, torsion free complex Kleinian group $\Gamma\subset U_+$, the parabolic part of $\Gamma$ is conjugate to one of the six options given by Theorem \ref{thm_descripcion_parte_parabolica}. According to Theorem \ref{thm_descomposicion_caso_noconmutativo}, $\Gamma_p$ correspond to the semidirect product of elements of the first two layers. Therefore, in order to give the full description of $\Gamma$, we have to check whether we can add one, two or three loxodromic elements of the third or fourth layer, taking into account that the rank of $\Gamma$ has to be at most four. In Table \ref{table:casos_rangos_abc} we summarize all these possible cases.

We denote by $N_3$ and $N_4$ the groups generated by the elements of the third and fourth layer respectively. Given $m,n$ not simultaneously $0$, we say that a discrete group $\Gamma\subset U_+$ belongs to case $kmn$ if $r(\Gamma_p)=k$, $r(N_3)=m$ and $r(N_4)=n$. Additionally, we say that $\Gamma$ belongs to subcase $kmn(j)$ if $\Gamma_p$ is conjugate to a group given by case (j) of Theorem \ref{thm_descripcion_parte_parabolica}.


Observe that purely parabolic cases 400, 300, 200 and 100 have already been described in \cite{barrera2022} (these are the cases listed in Theorem \ref{thm_descripcion_parte_parabolica}). 

\begin{table}[H]
\centering
\scriptsize
  \begin{tabular}{ | c | c | c | c || c | c | c | c || c | c | c | c ||}
    \hline
    Case & r$(\Gamma_p)$ & r$(N_3)$ & r$(N_4)$ & Case & r$(\Gamma_p)$ & r$(N_3)$ & r$(N_4)$ & Case & r$(\Gamma_p)$ & r$(N_3)$ & r$(N_4)$ \\ \hline
    400 & 4 & 0 & 0 & 202 & 2 & 0 & 2 & 111 & 1 & 1 & 1 \\ \hline
    310 & 3 & 1 & 0 & 201 & 2 & 0 & 1 & 110 & 1 & 1 & 0 \\ \hline
    301 & 3 & 0 & 1 & 200 & 2 & 0 & 0 & 103 & 1 & 0 & 3 \\ \hline
	300 & 3 & 0 & 0 & 130 & 2 & 0 & 0 & 102 & 1 & 0 & 2 \\ \hline    
	220 & 2 & 2 & 0 & 121 & 1 & 2 & 1 & 101 & 1 & 0 & 1 \\ \hline    
	211 & 2 & 1 & 1 & 120 & 1 & 2 & 0 & 100 & 1 & 0 & 0 \\ \hline    
	210 & 2 & 1 & 0 & 112 & 1 & 1 & 2 & & & & \\ \hline	
  \end{tabular}
\caption{All the possibilities for the rank of each layer of $\Gamma$.}
\label{table:casos_rangos_abc}
\end{table}

We will separately describe subcases $kmn(1),...,kmn(6)$ in each of the following subsections.


We will repeatedly use the arguments given in the following lemma in the following subsections. The proof of the lemma is straightforward.

\begin{lem}\label{obs_los_dos_calculos}
Let $\Gamma\subset U_+$ a discrete non-abelian group and let $\Gamma_p\subset\Gamma$ be the subgroup generated by all the parabolic elements of $\Gamma$.
	\begin{enumerate}[(i)]
	\item Using the decomposition given in Theorem \ref{thm_descomposicion_caso_noconmutativo}, if $\Gamma=\prodint{\gamma_1}\rtimes\cdots\rtimes\prodint{\gamma_N}$ for $1<m\leq N$, it follows that $\gamma_m^{-1}\gamma_1...\gamma_{m-1}\gamma_m=\gamma_1^{p_1}...\gamma_{m-1}^{p_{m-1}}$ for some $p_1,...,p_{m-1}\in\Z$. This is a consequence of the normality condition of the semidirect product 
	\item Let $\gamma_1,\gamma_2\in\Gamma$ be two loxodromic elements, then $\corchetes{\gamma_1,\gamma_2}\in\Gamma_p$.
	\end{enumerate}
\end{lem}

\subsection{The parabolic is an elliptic group}

The next lemma states that subgroups of $U_+$ whose parabolic part is an elliptic group are abelian. Representations of abelian discrete subgroups of $U_+$ have already been described in Section 4 of \cite{toledo2021dynamics}. 

\begin{lem}\label{lem:elliptic_abelian}
Let $\Gamma\subset U_+$ be a complex Kleinian group such that $\Gamma_p$ is conjugate to an elliptic group $\Ell(W,\mu)$ for some subgroup $W\subset\C$, and a group morphism $\mu$ as in the case (1) of Theorem \ref{thm_descripcion_parte_parabolica}. Then $\Gamma$ is abelian.
\end{lem}

\begin{proof}
We can assume that $\Gamma_p=\Ell(W,\mu)$, on the other hand, $\valorabs{\mu^{-3}(w)}=1$ for any $w\in W$. Then $\lambda_{23}(\gamma)$ has infinite order for any $\gamma\in\Gamma_p$, it follows from Lemma 4.13 of \cite{barrera2022} that $\Gamma$ is abelian. 
\end{proof}

\subsection{The parabolic part is a torus group}

The following lemma discards subcases 110(2), 111(2), 112(2), 120(2), 121(2), 130(2), 210(2), 211(2), 220(2) and 310(2). 

\begin{lem}\label{lem_casos_kmn2_noexisten}
Let $\Gamma\subset U_+$ be a complex Kleinian group such that $\Gamma_p$ is conjugate to a torus group $\mathcal{T}(W)$ for some discrete group $W$ as in the case (2) of Theorem \ref{thm_descripcion_parte_parabolica}. Then $\Gamma$ cannot contain elements in its third layer.
\end{lem}

\begin{proof}
Let us assume that $\Gamma_p=\mathcal{T}(W)$ for some additive subgroup $W\subset\C^2$, then $\Gamma_p = \SET{h_{x,y}}{(x,y)\in W}$.
Let us suppose that there exists a loxodromic element $\gamma\in N_3$, 
	$$\gamma=\corchetes{\begin{array}{ccc}
	\alpha^{-2} & \gamma_{12} & \gamma_{13}\\
	0 & \alpha & \gamma_{23}\\
	0 & 0 & \alpha\\
	\end{array}},\; \begin{array}{l}
	\text{for some }\alpha,\gamma_{23}\in\C^\ast,\;\valorabs{\alpha}\neq 1\\
	\gamma_{12},\gamma_{13}\in\C.
	\end{array}
	$$
Conjugating $\Gamma$ by an appropriate element of $\PSL$, we can assume that $\gamma_{12}=0$ and $\gamma_{23}=1$. Assume without loss of generality that $\valorabs{\alpha}>1$, and consider the sequence of distinct elements $\set{\gamma^k h_{x,y}\gamma^{-k}}\subset\Gamma$, it follows that $\gamma^k h_{x,y}\gamma^{-k}\rightarrow h_{\alpha^{-3k}x,y} \in\PSL$, contradicting that $\Gamma$ is discrete.\\
\end{proof}

The following lemma discards subcases 101(2), 102(2) and 103(2).

\begin{lem}\label{lem_casos_kmn2_noexisten_2}
Let $\Gamma\subset U_+$ be a complex Kleinian group such that $\Gamma_p$ is conjugate to a torus group $\mathcal{T}(W)$ for some discrete group $W$ as in the case (2) of Theorem \ref{thm_descripcion_parte_parabolica}. If $\mathcal{T}(W)=\prodint{h_{x,y}}$ with $y\neq 0$, $\Gamma$ contains no loxodromic elements in its fourth layer. 
\end{lem}

\begin{proof}
Assume that $\Gamma_p=\prodint{h=h_{x,y}}$, for some $x,y\in\C$ such that $y\neq 0$. Let us suppose that there exists a loxodromic element $\gamma=\corchetes{\gamma_{ij}}\in N_4$ that, up to conjugation, has the form
	$$\gamma=\corchetes{\begin{array}{ccc}
	\alpha & 0 & \gamma_{13}\\
	0 & \beta & \gamma_{23}\\
	0 & 0 & \alpha^{-1}\beta^{-1}\\
	\end{array}},\;
	\begin{array}{l}
	\alpha,\beta\in\C^\ast\\	
	\alpha\beta^2\neq 1.	
	\end{array}$$
Since $\prodint{h}$ is a normal subgroup of $\prodint{h,\gamma}$, we have that $\gamma h\gamma^{-1}=g^n$ for some $n\in\Z$. Comparing the entries 13 and 23 in the previous equation yields $\alpha^2\beta x = nx$ and $\alpha\beta^2 y = ny$. 

If $x\neq 0$, $y\neq 0$, the equation above implies that $\alpha=\beta$, which in turn implies that $\gamma$ is a complex homothety, contradicting Proposition \ref{cor_HC3_no_hay_en_no_conmutativos}. If $x=0$, $y\neq 0$, the sequence $\tau_k:=\gamma^k h \gamma^{-k} = h_{0,\zeta}$, where $\zeta=(\alpha\beta^2)^k y$, contains a subsequence of distinct elements $\set{\tau_{k_j}}$ such that, either $\tau_{k_j}\rightarrow\id$ or $\tau_{k_j}\rightarrow h$ (depending on whether $\valorabs{\alpha\beta^2}\neq 1$ or $\valorabs{\alpha\beta^2}=1$). 
\end{proof}

In the previous lemma, we omit the case $y=0$ because it will be covered by Lemma \ref{lem_kmn4_combinaciones_Gamma5}. The following lemma describes case 301(2).

\begin{lem}\label{lem_casos_kmn2_unico}
Let $\Gamma\subset U_+$ be a complex Kleinian group such that $\Gamma_p$ is conjugate to a torus group $\mathcal{T}(W)$ for some discrete group $W$ as in the case (2) of Theorem \ref{thm_descripcion_parte_parabolica}. If $r(\Gamma_p)=3$ and $r(N_4)=1$ then one generator of $N_4$ is 
	$$\gamma=\text{Diag}\parentesis{\alpha^2\beta,\alpha\beta^2,1},\;
	\text{with }\alpha\neq 1\text{ and }\alpha\beta^2\nin\R.	
	$$
\end{lem}

\begin{proof}
One can assume that $\Gamma_p=\mathcal{T}(W)$. Take three generators of $\mathcal{T}(W)$ denoted by $h_i=h_{x_i,y_i}$, $i=1,2,3$. Let $\gamma\in N_4$ be a loxodromic generator of $N_4$, up to a suitable conjugation, one can assume that $\gamma=\text{Diag}\parentesis{\alpha^2\beta,\alpha\beta^2,1}$, for some $\alpha,\beta\in\C^\ast$. Since $\Gamma_p$ is normal in $\Gamma$ we have that $\gamma h_i\gamma^{-1}=h_1^{k_{i,1}}h_2^{k_{i,2}}h_3^{k_{i,3}}$ for some $k_{i,j}\in\Z$, $i=1,2,3$. Let us define $A=\corchetes{k_{i,j}}\subset \text{PSL}(3,\Z)$, then
	$$A\corchetes{\begin{array}{c}
	x_1\\
	x_2\\
	x_3\\
	\end{array}}=\alpha^2\beta\corchetes{\begin{array}{c}
	x_1\\
	x_2\\
	x_3\\
	\end{array}},\;\;\;\;A\corchetes{\begin{array}{c}
	y_1\\
	y_2\\
	y_3\\
	\end{array}}=\alpha\beta^2\corchetes{\begin{array}{c}
	y_1\\
	y_2\\
	y_3\\
	\end{array}}.$$
This means that $\alpha^2\beta$ and $\alpha\beta^2$ are eigenvalues of $A$, with respective eigenvectors $\corchetes{x_1:x_2:x_3}$ and $\corchetes{y_1:y_2:y_3}$. A direct calculation shows that $\alpha\neq 1$ and $\alpha\beta^2\nin\R$. Using (1) of Theorem 4.4 of \cite{cs2014} we conclude the proof. 
\end{proof}

\subsection{The parabolic part is an abelian Kodaira group}

The following lemma discards all subcases kmn(3).

\begin{lem}\label{lem_Gamma_RLW_no_tiene_lox}
Let $\Gamma\subset U_+$ be a complex Kleinian group such that $\Gamma_P$ is conjugate to an abelian Kodaira group $\text{Kod}_0(W,R,L)$ for some $R,L,W$ as in the case (3) of Theorem \ref{thm_descripcion_parte_parabolica}. Then $\Gamma$ cannot contain loxodromic elements.
\end{lem}

\begin{proof}
Assume that $\Gamma_p=\text{Kod}_0(W,R,L)$. Let us suppose that $\Gamma$ contain a loxodromic element $\gamma\in\Gamma$. If $\gamma\in N_3$ then
	$$\gamma=\corchetes{\begin{array}{ccc}
	\alpha^{-2} & \gamma_{12} & \gamma_{13}\\
	0 & \alpha & \gamma_{23}\\
	0 & 0 & \alpha\\
	\end{array}},\;
	\begin{array}{l}
	\text{for some }\alpha\in\C^\ast,\;	\valorabs{\alpha}\neq 1\\
	\gamma_{12},\gamma_{13}\in\C\text{ and }\gamma_{23}\in\C^\ast.
	\end{array}$$
Since $\Gamma_p$ is a normal subgroup of $\Gamma$, we have $\gamma g \gamma^{-1}\in\Gamma_p$ for any $g\in\Gamma_p$. If 
	\begin{equation}\label{eq_dem_lem_Gamma_RLW_no_tiene_lox_1}		
	g=\corchetes{\begin{array}{ccc}
	1 & x & L(x)+\frac{x^{2}}{2}+w\\
	0 & 1 & x\\
	0 & 0 & 1\\
	\end{array}},\text{ then }
	\end{equation}
	$$\gamma g\gamma^{-1}=\corchetes{\begin{array}{ccc}
	1 & \alpha^{-3}x & \alpha\parentesis{L(x)+\frac{x^{2}}{2}+w}-\frac{\gamma_{23}x}{1-\alpha^3}\\
	0 & 1 & x\\
	0 & 0 & 1\\
	\end{array}},$$
and, since $\gamma g \gamma^{-1}\in\Gamma_p$, then $\alpha^{-3}x=x$, which is impossible, given that $\valorabs{\alpha}\neq 1$. Now, if $\gamma\in N_4$, then
	$$\gamma=\corchetes{\begin{array}{ccc}
	\alpha & \gamma_{12} & \gamma_{13}\\
	0 & \beta & \gamma_{23}\\
	0 & 0 & \alpha^{-1}\beta^{-1}\\
	\end{array}}$$
with $\alpha,\beta\in\C^\ast$, $\alpha^{-1}\beta^{-1}\neq \beta$ (or equivalently, $\alpha\beta^{2}\neq 1$). Again, since $\Gamma_p$ is a normal subgroup of $\Gamma$, we have $\gamma g \gamma^{-1}\in\Gamma_p$ for any $g\in\Gamma_p$. If $g$ is given by (\ref{eq_dem_lem_Gamma_RLW_no_tiene_lox_1}), then
	$$\gamma g\gamma^{-1}=\corchetes{\begin{array}{ccc}
	1 & \alpha\beta^{-1}x & \alpha^{2}\beta\parentesis{L(x)+\frac{x^{2}}{2}+w}+\alpha\gamma_{23}x(\beta -\alpha)\\
	0 & 1 & \alpha\beta^{2}x\\
	0 & 0 & 1\\
	\end{array}},$$
and, since $\gamma g \gamma^{-1}\in\Gamma_p$, then $\alpha\beta^{-1}=1$ and $\alpha\beta^{2}=1$, but this last condition is a contradiction. This contradictions prove the lemma. 
\end{proof}

\subsection{The parabolic part is a dual torus group}

Now we deal with subcases kmn(4). The following proposition describes the form of $\Gamma_p$, up to conjugation. Its proof is straight-forward, and its based on Lemma 2.3 and 2.5 of \cite{barrera2022}. 

\begin{prop}\label{prop_casos_kmn4_conjugaciones}
Let $\Gamma\subset U_+$ be a complex Kleinian group such that $\Gamma_p$ is conjugate to a dual torus group $\mathcal{T}^{\ast}(W)$ for some discrete additive subgroup $W\subset\C^2$, as in (4) of Theorem \ref{thm_descripcion_parte_parabolica}. Then $\Gamma_p$ is conjugate to one of the following groups:

\begin{multicols}{2}
\begin{enumerate}[(i)]
\item $\Gamma_1=\prodint{g_{1,0},g_{0,1}}$.
\item $\Gamma_2=\prodint{g_{0,1},g_{0,y}}$, with $y\nin\R$.
\item $\Gamma_3=\prodint{g_{1,0},g_{x,0}}$, with $x\nin\R$.
\item $\Gamma_4=\prodint{g_{1,0}}$.
\item $\Gamma_5=\prodint{g_{0,1}}$.
\end{enumerate}
\end{multicols}

\end{prop}

In the context of Lemma 2.3 and 2.5 of \cite{barrera2022}, $\Gamma_1$ is a dual torus group of type I, and the rest are dual torus group of type II.

\medskip

The following lemmas study the different combinations of $\Gamma_p$ and elements of the third and fourth layer of $\Gamma$. In all of these lemmas, $\Gamma$ is a discrete subgroup of $U_+$, and $\Gamma_p$ denotes the parabolic part of $\Gamma$.\\

\begin{lem}\label{lem_kmn4_combinaciones_Gamma1}
Suppose that $\Gamma_p$ is conjugate to $\Gamma_1=\prodint{g_{1,0},g_{0,1}}$.
	\begin{enumerate}[(i)]
	\item If $\rank(N_3)=1$, one generator of $N_3$ is given by
		\begin{equation}\label{eq_lem_kmn4_combinaciones_Gamma1_2}
		\gamma=\corchetes{\begin{array}{ccc}
	p & \gamma_{12} & \gamma_{13}\\
	0 & 1 & \frac{q}{p}\\
	0 & 0 & 1\\
	\end{array}},\;
		\begin{array}{l}
		\text{ for }p\in\Z\setminus\set{-1,0,1}\text{, }q\in\Z\setminus\set{0}\\
		\gamma_{12},\gamma_{13}\in\C.			
		\end{array}
		\end{equation}
	\item If $\rank(N_3)=2$, one generator of $N_3$ is given by (\ref{eq_lem_kmn4_combinaciones_Gamma1_2}), and the other is given by
		\begin{equation}\label{eq_lem_kmn4_combinaciones_Gamma1_4}\mu=\corchetes{\begin{array}{ccc}
	p_2 & \mu_{12} & \mu_{13}\\
	0 & 1 & \frac{q_2}{p_2}\\
	0 & 0 & 1\\
	\end{array}},
	\end{equation}
	with $p_2\in\Z\setminus\set{-1,0,1}$, $q_2\in\Z\setminus\set{0}$, $\mu_{12},\mu_{13}\in\Q\setminus\set{0}$ and
		\begin{align*}
		\mu_{12} &= -\frac{p_1 p_2(m+jp_1p_2)}{(1-p_1)(p_1q_2+p_2q_1)}\\
		\mu_{13} &= \frac{p_2\parentesis{mq_1+jp_1^2(p_2q_1-(1-p_1)q_2)}}{(1-p_1)^2(p_1q_2+p_2q_1)}
		\end{align*}
		where $j,m\in\Z$ are integers such that $p_1 p_2+p_1q_2 \mid m+jp_1p_2$ and $p_1 p_2+p_1q_2 \mid mp_1q_2 -jp_1p_2^2q_1$.
		
	\item If $\rank(N_3)=0$ and $\rank(N_4)=1$, one generator of $N_4$ is given by
		\begin{equation}\label{eq_lem_kmn4_combinaciones_Gamma1_3}
		\gamma=\corchetes{\begin{array}{ccc}
	pq & \gamma_{12} & \gamma_{13}\\
	0 & q & r\\
	0 & 0 & p\\
	\end{array}},
		\begin{array}{l}
		\text{ for }p,q\in\Z\setminus\set{0}\text{ such that }\\
		\gamma\text{ is loxodromic,  }\\
		r\in\Z,\;\gamma_{12},\gamma_{13}\in\C.		
		\end{array}				
		\end{equation}
	\item It's not possible to have $\rank(N_4)=2$.
	\item If $\rank(N_3)=1$, then there are no elements in the fourth layer $N_4$.
	\end{enumerate}
\end{lem}

\begin{proof}
We can assume that $\Gamma_p=\Gamma_1$. Let $g\in\Gamma_p$ be an element with the form $g=g_{n,m}$ for some $n,m\in\Z$. We prove each conclusion separately:
	\begin{enumerate}[(i)]
	\item Let $\gamma\in N_3$ be a generator of the third layer of $\Gamma$. Then
	$$\gamma=\corchetes{\begin{array}{ccc}
	\alpha^{-2} & \gamma_{12} & \gamma_{13}\\
	0 & \alpha & \gamma_{23}\\
	0 & 0 & \alpha\\
	\end{array}}\text{, for some }\gamma_{23},\alpha\in\C^\ast,\;\valorabs{\alpha}\neq 1.$$
	Since $\Gamma_p$ is a normal subgroup of $\Gamma$, then $\gamma g \gamma^{-1}\in\Gamma_p$, which means that $\alpha^{-3}n,\alpha^{-4}(m\alpha-n\gamma_{23})\in\Z$ for any $n,m\in\Z$ (see (i) of Lemma \ref{obs_los_dos_calculos}). In particular, for $n=0,m=1$ we get $\alpha^{-3}=p$ for some $p\in\Z\setminus\set{0}$, for $n=1,m=0$ we get $\gamma_{23}=q p^{-\frac{4}{3}}$ for some $q\in\Z\setminus\set{0}$. Then
		$$\gamma=\corchetes{\begin{array}{ccc}
	p^{\frac{2}{3}} & \gamma_{12} & \gamma_{13}\\
	0 & p^{-\frac{1}{3}} & qp^{-\frac{4}{3}}\\
	0 & 0 & p^{-\frac{1}{3}}\\
	\end{array}}=\corchetes{\begin{array}{ccc}
	p & p^{\frac{1}{3}}\gamma_{12} & p^{\frac{1}{3}}\gamma_{13}\\
	0 & 1 & \frac{q}{p}\\
	0 & 0 & 1\\
	\end{array}},$$
and since $\gamma_{12}$ and $\gamma_{13}$ are arbitrary, we get (\ref{eq_lem_kmn4_combinaciones_Gamma1_2}).
	\item Let $\mu=\corchetes{\mu_{ij}}$ be a second generator of $N_3$. Conjugating $\Gamma$ by an adequate element of $\PSL$, we can assume that 
		$$\gamma=\corchetes{\begin{array}{ccc}
	p_1 & 0 & 0\\
	0 & 1 & \frac{q_1}{p_1}\\
	0 & 0 & 1\\
	\end{array}}.$$
	Using the proof of (i), $\mu$ has the form given by (\ref{eq_lem_kmn4_combinaciones_Gamma1_4}). Observe that if $\corchetes{\mu,\gamma}=\id$ then $\mu_{12}=\mu_{13}=0$ and then $\Gamma$ would be abelian. Then, we can assume that $\corchetes{\mu,\gamma}\in\Gamma_p\setminus\set{\id}$ (see (ii) of Observation \ref{obs_los_dos_calculos}), comparing the entry 13 of the last expression yields
		\begin{equation}\label{eq_dem_lem_kmn4_combinaciones_Gamma1_4}
		\frac{q\mu_{12}+(1-p_1)p_1\mu_{13}}{p_1^2p_2}=j
		\end{equation}
	for some $j\in\Z\setminus\set{0}$. Since $\Gamma_p\rtimes\prodint{\gamma}$ is normal in $\Gamma$ we have
		$$\mu\gamma\mu^{-1}\in \Gamma_p\rtimes\prodint{\gamma}=\SET{\corchetes{\begin{array}{ccc}
	p_1^k & n & m+\frac{knq_1}{p_1}\\
	0 & 1 & k\frac{q_1}{p_1}\\
	0 & 0 & 1\\
	\end{array}}}{m,n\in\Z}.$$		
		Comparing the entries 12 and 13 in the last expression yields		
		\begin{align}
		\mu_{12}(1-p_1)&=n \nonumber \\
		\mu_{12}\parentesis{\frac{q_1}{p_1}-\frac{q_2}{p_2}}+\mu_{12}\frac{p_1q_2}{p_2}+\mu_{13}(1-p_1)&=m+n\frac{q_1}{p_1}\label{eq_dem_lem_kmn4_combinaciones_Gamma1_3}
		\end{align}
	for some $m,n\in\Z$. Combining both expresions in (\ref{eq_dem_lem_kmn4_combinaciones_Gamma1_3}) we get
		\begin{equation}\label{eq_dem_lem_kmn4_combinaciones_Gamma1_5}
		mp_2+\parentesis{(1-p_1)q_2-p_2q_1}\mu_{12}+p_2\mu_{13}(1-p_1)=0.
		\end{equation}		 
		Solving (\ref{eq_dem_lem_kmn4_combinaciones_Gamma1_4}) and (\ref{eq_dem_lem_kmn4_combinaciones_Gamma1_5}) we get the desired expressions for $\mu_{12}$ and $\mu_{13}$. 
	\item If $\gamma\in N_4$ then 
		\begin{equation}\label{eq_dem_lem_kmn4_combinaciones_Gamma1_1}
		\gamma=\corchetes{\begin{array}{ccc}
	\alpha & \gamma_{12} & \gamma_{13}\\
	0 & \beta & \gamma_{23}\\
	0 & 0 & \alpha^{-1}\beta^{-1}\\
	\end{array}}
	\end{equation}				
	for some $\alpha,\beta\in\C^{\ast}$ such that $\gamma$ is loxodromic and $\alpha\neq\beta^2$, $\gamma_{12},\gamma_{13},\gamma_{23}\in\C$. Using (i) of Lemma \ref{obs_los_dos_calculos} it follows that $\gamma g \gamma^{-1}\in\Gamma_p$, which means that $\alpha\beta^{-1},\alpha^{2}(m\beta-n\gamma_{23})\in\Z$ for any $m,n\in\Z$. If $n=0,m=1$ then $\alpha^2\beta=q$ for some $q\in\Z$. If $n=1,m=0$ then $\alpha\beta^{-1}=p\in\Z$ and $\alpha^2\gamma_{23}=r\in\Z$. All this together yields (\ref{eq_lem_kmn4_combinaciones_Gamma1_3}).
	\item Assume that $\rank(N_4)=2$ and let $\gamma$, $\mu$ be two generators of $N_4$. Using an adequate conjugation we can assume that 
		$$\gamma=\corchetes{\begin{array}{ccc}
	q_1p_1 & 0 & \gamma_{13}\\
	0 & q_1 & 0\\
	0 & 0 & p_1\\
	\end{array}},\;\;\;\;\mu=\corchetes{\begin{array}{ccc}
	q_2p_2 & \mu_{12} & \mu_{13}\\
	0 & q_2 & \mu_{23}\\
	0 & 0 & p_2\\
	\end{array}}.$$
	Then, using normality and Theorem \ref{thm_descomposicion_caso_noconmutativo} we have $\mu\gamma\mu^{-1}\in\Gamma_p\rtimes\prodint{\gamma}$. Comparing the entries 12, we have
		$$\frac{q_1\mu_{12}(1-p_1)}{q_2}=-q_1\mu_{12}(1-p_1)$$
	and therefore, $q_2=-1$. Then $\lambda_{13}(\mu)$ is a torsion element in $\lambda_{13}$, this contradiction proves that we cannot have $\rank(N_4)=2$.
	\item Let $\gamma$ be a generator of $N_3$, using an adequate conjugation we can assume that 
		$$\gamma =\corchetes{\begin{array}{ccc}
	p & 0 & \gamma_{13}\\
	0 & 1 & \frac{q}{p}\\
	0 & 0 & 1\\		
	\end{array}},
	\begin{array}{l}
	\text{ with }p,q\in\Z\setminus\set{0}\\
	\valorabs{p}\neq 1.	
	\end{array}		
	$$
Let us assume that there is an element $\mu\in N_4$, then
		$$\mu =\corchetes{\begin{array}{ccc}
	\alpha & \mu_{12} & \mu_{13}\\
	0 & \beta & \mu_{23}\\
	0 & 0 & \alpha^{-1}\beta^{-1}\\		
	\end{array}},$$
with $\alpha\beta^{2}\neq 1$. Since $\corchetes{\gamma,\mu}\in\Gamma_p$, then comparing the entries 23 yields $q(1-\alpha\beta^2)p^{-1}\alpha^{-1}\beta^{-2}=0$. This means that, either $q=0$ or $\alpha\beta^2=1$. Both conclusions contradict the hypotheses. Thus, there are no elements in $N_4$.
	\end{enumerate}
\end{proof}

\begin{lem}\label{lem_kmn4_combinaciones_Gamma2}
Suppose that $\Gamma_p$ is conjugate to $\Gamma_2=\prodint{g_{0,1},g_{0,y}}$, with $y\nin\R$.
	\begin{enumerate}[(i)]
	\item If $\rank(N_3)=1$ then one generator of $N_3$ is given by
		\begin{equation}\label{eq_lem_kmn4_combinaciones_Gamma2_2}
		\gamma=\corchetes{\begin{array}{ccc}
	\parentesis{p+q\text{Re}(y)}+iq\text{Im}(y) & \gamma_{12} & \gamma_{13}\\
	0 & 1 & \gamma_{23}\\
	0 & 0 & 1\\
	\end{array}},
		\end{equation}
	for $p,q\in\Z$ such that $\valorabs{p}+\valorabs{q}\neq 0$ and $\valorabs{\parentesis{p+q\text{Re}(y)}+iq\text{Im}(y)}\neq 1$, $\gamma_{12},\gamma_{23}\in\C$ and $\gamma_{23}\in\C^\ast$.
	\item If $\rank(N_3)=0$ and $\rank(N_4)=1$ then one generator of $N_4$ is given by
		\begin{equation}\label{eq_lem_kmn4_combinaciones_Gamma2_3}
		\gamma=\corchetes{\begin{array}{ccc}
	\alpha & \gamma_{12} & \gamma_{13}\\
	0 & \alpha^{-2}z_{p,q} & \gamma_{23}\\
	0 & 0 & \alpha z_{p,q}^{-1}\\
	\end{array}},
		\end{equation}
	for $\beta\in\C^{\ast}$, $z_{p,q}=\parentesis{p+q\text{Re}(y)}+iq\text{Im}(y)$, $\valorabs{p}+\valorabs{q}\neq 0$, $\gamma_{12},\gamma_{23},\gamma_{23}\in\C$. Finally, $\gamma$ must be loxodromic. 
	\end{enumerate}
\end{lem}

\begin{proof}
We can assume that $\Gamma_p=\Gamma_2$. Let $g=g_{0,n+my}\in\Gamma_p$, with $n,m\in\Z$. We prove each conclusion separately:
	\begin{enumerate}[(i)]
	\item Let $\gamma\in N_3$ be a generator of the third layer of $\Gamma$. Then
	$$\gamma=\corchetes{\begin{array}{ccc}
	\alpha^{-2} & \gamma_{12} & \gamma_{13}\\
	0 & \alpha & \gamma_{23}\\
	0 & 0 & \alpha\\
	\end{array}}\in N_3,\;
	\begin{array}{l}
	\text{ for some }\gamma_{23},\alpha\in\C^\ast\\
	\valorabs{\alpha}\neq 1.
	\end{array}		
	$$
	Since $\Gamma_p$ is a normal subgroup of $\Gamma$, then $\gamma g \gamma^{-1}\in\Gamma_p$ and then, for any $n,m\in\Z$,
		\begin{equation}\label{eq_dem_lem_kmn4_combinaciones_Gamma2_1}
		\alpha^{-3}(n+my)=p+qy
		\end{equation}		 
		for some $p,q\in\Z$. Denote $y=a+bi$ and $\alpha^{-3}=r(\cos\theta+i\sin\theta)$. In particular, (\ref{eq_dem_lem_kmn4_combinaciones_Gamma2_1}) holds for $n=1$, $m=0$. Then, (\ref{eq_dem_lem_kmn4_combinaciones_Gamma2_1}) becomes $r\cos\theta=p+qa$, $r\sin\theta= q b$. Which means that $\text{Re}\parentesis{\alpha^{-3}}=p+qa$ and $\text{Im}\parentesis{\alpha^{-3}}=qb$ and therefore, $\alpha^{-3}=\parentesis{p+q\text{Re}(y)}+iq\text{Im}(y)$ for $p,q\in\Z$ such that $\valorabs{p}+\valorabs{q}\neq 0$ (otherwise, $\alpha=0$) and $\valorabs{\parentesis{p+q\text{Re}(y)}+iq\text{Im}(y)}\neq 1$ (otherwise, $\gamma$ would not be loxodromic). Since we can re-write $\gamma$ as
			$$\gamma=\corchetes{\begin{array}{ccc}
	\alpha^{-3} & \alpha^{-1}\gamma_{12} & \alpha^{-1}\gamma_{13}\\
	0 & 1 & \alpha^{-1}\gamma_{23}\\
	0 & 0 & 1\\
	\end{array}},$$
	we get (\ref{eq_lem_kmn4_combinaciones_Gamma2_2}).
	
	\item If $\gamma\in N_4$, then $\gamma$ has the form given by (\ref{eq_dem_lem_kmn4_combinaciones_Gamma1_1}). Since $\Gamma_p$ is normal in $\Gamma$, $\gamma g\gamma^{-1}\in \Gamma$, and this implies that $\alpha^2\beta(n+my)=\tilde{n}+\tilde{m}y$, for some $\tilde{n},\tilde{m}\in\Z$. The same argument and calculations used in the proof of (i) above show that $\alpha^2\beta=\parentesis{p+q\text{Re}(y)}+iq\text{Im}(y)$. Substituting this in the form of $\gamma$ yields (\ref{eq_lem_kmn4_combinaciones_Gamma2_3}).		  
	\end{enumerate}
\end{proof}

\begin{lem}\label{lem_kmn4_combinaciones_Gamma3}
Suppose that $\Gamma_p$ is conjugate to $\Gamma_3=\prodint{g_{1,0},g_{x,0}}$, with $x\nin\R$.
	\begin{enumerate}[(i)]
	\item $\Gamma$ cannot contain elements in its third layer.
	\item If $\rank(N_4)=1$ then one generator of $N_4$ is given by
		\begin{equation}\label{eq_lem_kmn4_combinaciones_Gamma3_2}
		\gamma=\corchetes{\begin{array}{ccc}
	p+qx & \gamma_{12} & \gamma_{13}\\
	0 & 1 & 0\\
	0 & 0 & (p+qx)^{-1}\beta^{-3}\\
	\end{array}},
		\begin{array}{l}
		p,q\in\Z;\;\valorabs{p}+\valorabs{q}\neq 0\\
		\beta\in\C^\ast,\;\gamma_{12}, \gamma_{13}\in\C,
		\end{array}			
		\end{equation}
	such that $\gamma$ is loxodromic. Also
		\begin{equation}\label{eq_lem_kmn4_combinaciones_Gamma3_3}
		p-q\valorabs{x}^2 \in\Z,\;\;\;\;p+q+2q\text{Re}(x)\in\Z.
		\end{equation}
	\item If $\rank(N_4)=2$ then one generator is given by (ii) and the other is given by
		\begin{equation}\label{eq_lem_kmn4_combinaciones_Gamma3_4}
		\mu=\corchetes{\begin{array}{ccc}
	p_2+q_2x & \mu_{12} & 0\\
	0 & 1 & 0\\
	0 & 0 & (p_2+q_2x)^{-1}\beta_2^{-3}\\
	\end{array}},
		\end{equation}
	where $p_2,q_2\in\Z$ such that $\valorabs{p_2}+\valorabs{q_2}\neq 0$ and both integers satisfy (\ref{eq_lem_kmn4_combinaciones_Gamma3_3}), $\beta_2\in\C^\ast$ is such that $\gamma$ is loxodromic and $p_2-q\valorabs{x}^2\in\Z$. Furthermore, $\mu_{12}\in\C$ such that
		\begin{equation}\label{eq_lem_kmn4_combinaciones_Gamma3_5}
		\mu_{12}=\parentesis{n+mx-(x+1)(p_2+q_2 x)}\parentesis{1-(p_1+q_1 x)}^{-1}
		\end{equation} 
		for some $n,m\in\Z$ such that
		\begin{equation}\label{eq_lem_kmn4_combinaciones_Gamma3_6}
		\mu_{12}\frac{p_1+q_1 x-1}{(p_1+q_1 x)(p_2+q_2 x)}=k_1+k_2 x,\;\;\;k_1,k_2\in\Z.
		\end{equation}
	\end{enumerate}
\end{lem}

\begin{proof}
We can assume that $\Gamma_p=\Gamma_2$. Let $g=g_{n+mx}\in\Gamma_p$. We prove each conclusion separately:
	\begin{enumerate}[(i)]
	\item The argument is the same used in the first conclusion of Lemma \ref{lem_kmn4_combinaciones_Gamma4}.
	\item If $\gamma\in N_4$, then $\gamma$ has the form given by (\ref{eq_dem_lem_kmn4_combinaciones_Gamma1_1}). Since $\Gamma_p$ is normal in $\Gamma$, $\gamma g\gamma^{-1}\in \Gamma_p$, and this implies that $\alpha\beta^{-1}(n+mx)=\tilde{n}+\tilde{m}x$, and $-\alpha^2\gamma_{23}(n+mx)=0$, for some $\tilde{n},\tilde{m}\in\Z$. Then $\gamma_{23}=0$, and using the same argument and calculations as in the first part of the proof of Lemma \ref{lem_kmn4_combinaciones_Gamma2} we have
		$$\alpha\beta^{-1}=\parentesis{p+q\text{Re}(x)}+iq\text{Im}(x).$$
	Substituting this in (\ref{eq_dem_lem_kmn4_combinaciones_Gamma1_1}) yields (\ref{eq_lem_kmn4_combinaciones_Gamma3_2}). Again, since $\gamma g\gamma^{-1}\in \Gamma_p$ for $\gamma$ as in (\ref{eq_lem_kmn4_combinaciones_Gamma3_2}), comparing entries 12 yields 
		\begin{equation}\label{eq_dem_lem_kmn4_combinaciones_Gamma3_1}
		p + (p+q)x + qx^2=p'+q'x,
		\end{equation}				
		 for some $p',q'\in\Z$. To solve (\ref{eq_dem_lem_kmn4_combinaciones_Gamma3_1}), we write $x=\text{Re}(x)+i\text{Im}(x)$ and assume $\text{Im}(x)\neq 0$. From the equation corresponding to the imaginary part we get $(p+q)+2q\text{Re}(x)\in\Z$, substituting this into the equation corresponding to the real part we get $p-q\valorabs{x}^2\in\Z$. This verifies (\ref{eq_lem_kmn4_combinaciones_Gamma3_3}).
		 
		 \item Let $\gamma$ be a generator of $N_4$ given by (ii) and let $\mu$ be the second generator of $\mu$, then $\mu$ has the same form given by (ii), that is
		 	$$\mu=\corchetes{\begin{array}{ccc}
			p_2+q_2x & \mu_{12} & 0\\
			0 & 1 & 0\\
			0 & 0 & (p_2+q_2x)^{-1}\beta_2^{-3}\\
			\end{array}}.$$
		 We can conjugate $\Gamma$ by a suitable $A\in\PSL$ such that 
		 	$$A\gamma A^{-1}=\text{Diag}\parentesis{p_1+q_1 x ,1,\alpha^{-3}(p+q x)^{-1}}.$$ 
		 We have that, either $\corchetes{\mu,\gamma}=\id$ or $\corchetes{\mu,\gamma}=g_{k_1,k_2}$ for some $k_1,k_2\in\Z$. Both possibilities imply that $\mu_{13}=0$. If $\corchetes{\mu,\gamma}=\id$ then $\mu_{12}=0$, and then $\mu$ is diagonal. Since $\gamma$ is diagonal and the parabolic part is abelian, then $\Gamma$ would be abelian (and therefore, it has already been described in Section 4 of \cite{toledo2021dynamics}). Then, $\corchetes{\mu,\gamma}=g_{k_1,k_2}$ for some $k_1,k_2\in\Z$, this yields (\ref{eq_lem_kmn4_combinaciones_Gamma3_6}). On the other hand, since $\Gamma_p\rtimes\prodint{\gamma}$ is normal in $\Gamma$, it follows that $\mu g_{1,1} \gamma \mu^{-1}=g_{n,m} \gamma^k$ for some $n,m,k\in\Z$. A direct calculation shows that $k=1$, the remaining expression yields (\ref{eq_lem_kmn4_combinaciones_Gamma3_5}).
			 	 
	\end{enumerate}
\end{proof}

\begin{lem}\label{lem_kmn4_combinaciones_Gamma4}
Suppose that $\Gamma_p$ is conjugate to $\Gamma_4=\prodint{g_{1,0}}$.
	\begin{enumerate}[(i)]
	\item $\Gamma$ cannot contain elements of the third layer.
	\item If $\rank(N_4)=1$, one generator of $N_4$ is given by
		\begin{equation}\label{eq_lem_kmn4_combinaciones_Gamma4_2}
		\gamma=\corchetes{\begin{array}{ccc}
	p\alpha & \gamma_{12} & \gamma_{13}\\
	0 & \alpha & 0\\
	0 & 0 & p^{-1}\alpha^{-2}\\
	\end{array}},\;
		\begin{array}{l}
		p\in\Z\setminus\set{0,1},\;\alpha\in\C^{\ast}\\
		\valorabs{\alpha}\neq 1\\
		\gamma_{12},\gamma_{13}\in\C.
		\end{array}
		\end{equation}
	\item If $\rank(N_4)=2$, one generator of $N_4$ is given by (\ref{eq_lem_kmn4_combinaciones_Gamma4_2}), and the other satisfies one of the following conditions:
		\begin{enumerate}
		\item If $p^2\alpha^3=1$ then 
			$$\mu=\corchetes{\begin{array}{ccc}
	q\beta & \frac{\beta jpq}{1-p} & \mu_{13}\\
	0 & \beta & 0\\
	0 & 0 &q^{-1}\beta^{-2}\\
	\end{array}},\;\mu_{13}\in\C.$$
		\item If $p^2\alpha^3\neq 1$ then 
			$$\mu=\corchetes{\begin{array}{ccc}
	q\beta & \frac{\beta jpq}{1-p} & 0\\
	0 & \beta & 0\\
	0 & 0 &q^{-1}\beta^{-2}\\
	\end{array}}.$$
		\end{enumerate}
		In both cases, $\beta\in\C^\ast$, $j,q\in\Z$ and $q\neq 0$ such that $\mu$ is loxodromic.
	\item If $\rank(N_4)=3$, two generators of $N_4$ are given by (ii) and (iii) respectively, and the third generator $\eta$ is given by
		\begin{equation}\label{eq_lem_kmn4_combinaciones_Gamma4_3}
		\eta=\corchetes{\begin{array}{ccc}
	r\delta & \frac{\delta k pr}{1-p} & \eta_{13}\\
	0 & \delta & 0\\
	0 & 0 &q^{-1}\delta^{-2}\\
	\end{array}},
		\end{equation}
	where $\delta\in\C^\ast$, $r\in\Z\setminus\set{-1,0,1}$, $r\delta^3\neq 1$ and
		\begin{align*}
		\eta_{13} &= \frac{p\alpha^2\parentesis{\gamma_{13}(1-r^2\delta^3)-pq\alpha\beta^2\mu_{13}(1+r^2\delta^3)}}{r\delta^2(1-p^2q^2\alpha^3\beta^3)} \\
		k &= \frac{(1-p)\parentesis{\alpha(n+r)-\gamma_{12}(1-r)}-j_2p^2q\alpha(1-r)}{pr\alpha(1-pq)}
		\end{align*}
	for $n\in\Z$. Also, either
		$$
		k r(1-q) = jq(1-r)\;\; \text{ or } \;\;qr(1-p) \left.\right|\; p\parentesis{kr(1-q)-jq(1-r)},
		$$			 
	depending on whether $\mu$ and $\eta$ commute.
	\end{enumerate}		 
\end{lem}

\begin{proof}
We assume that $\Gamma_p=\Gamma_4$, we denote $g=g_{1,0}$. We prove each conclusion separately:
	\begin{enumerate}[(i)]
	\item Assume that $\Gamma$ contain an element of the third layer, 
		\begin{equation}\label{eq_dem_lem_kmn4_combinaciones_Gamma4_2}
		\gamma=\corchetes{\begin{array}{ccc}
	\alpha^{-2} & \gamma_{12} & \gamma_{13}\\
	0 & \alpha & \gamma_{23}\\
	0 & 0 & \alpha\\
	\end{array}}\in N_3,
	\begin{array}{l}
	\gamma_{23},\alpha\in\C^\ast\\
	\valorabs{\alpha}\neq 1
	\end{array}
	\end{equation}				
	Using (i) of Lemma \ref{obs_los_dos_calculos}, $\gamma g \gamma^{-1}\in\Gamma_p$. Comparing entries 13 yields $\gamma_{23}=0$, contradicting that $\gamma\in N_3$.
	\item If $\rank(N_4)=1$, let $\gamma\in N_4$ be an element given by
		\begin{equation}\label{eq_dem_lem_kmn4_combinaciones_Gamma4_1}
		\gamma=\corchetes{\begin{array}{ccc}
	\beta & \gamma_{12} & \gamma_{13}\\
	0 & \alpha & \gamma_{23}\\
	0 & 0 & \alpha^{-1}\beta^{-1}\\
	\end{array}},
		\begin{array}{l}
		\alpha^2\beta\neq 1,\;\alpha\neq\beta\\
		\gamma\text{ is loxodromic.}
		\end{array}
		\end{equation}				
	Since $\Gamma_p$ is a normal subgroup of $\Gamma$, then $\gamma g \gamma^{-1}\in\Gamma_p$. Comparing entries 13 yields $\gamma_{23}=0$ and, for $n=1$ in particular, $\beta\alpha^{-1}\in\Z$, in other words, $\beta=p\alpha$ for any $p\in\Z$. All of these conditions together with (\ref{eq_dem_lem_kmn4_combinaciones_Gamma4_1}) imply (\ref{eq_lem_kmn4_combinaciones_Gamma4_2}).
	\item Let $\gamma\in N_4$ be one generator of $N_4$ given by (\ref{eq_lem_kmn4_combinaciones_Gamma4_2}). Up to a suitable conjugation we can assume that $\gamma=\text{Diag}\parentesis{p\alpha,\alpha,p^{-1}\alpha^{-2}}$, with $p\in\Z\setminus\set{0}$ and $\alpha\in\C^\ast$. Using part (ii), we know that 
		$$\mu=\corchetes{\begin{array}{ccc}
	q\beta & \mu_{12} & \mu_{13}\\
	0 & \beta & 0\\
	0 & 0 & q^{-1}\beta^{-2}\\
	\end{array}}.$$
	Since $\corchetes{\mu,\gamma}\in\prodint{g}$, comparing the respective entries we have $\mu_{13}(1-p^2\alpha^3)=0$ and $\mu_{12}(1-p)=\beta jpq$ for some $j\in\Z$. Therefore
	\begin{equation}\label{eq_dem_lem_kmn4_combinaciones_Gamma4_03}
	\mu_{13} = 0\text{ or }p^2\alpha^3=1;\;\;\text{ and }\; \mu_{12} = \frac{\beta jpq}{1-p}. 
	\end{equation}		
	A direct calculation shows that these two conditions are equivalent to prove that $\Gamma_p\rtimes\prodint{\gamma}$ is normal in $\Gamma_p\rtimes\prodint{\gamma}\rtimes\prodint{\mu}$. Therefore the only restrictions for $\mu$ are (\ref{eq_dem_lem_kmn4_combinaciones_Gamma4_03}). This proves this part of the lemma.
	
	\item If $\rank(N_4)=3$, we denote by $\gamma$ and $\mu$ the first two generators of $N_4$, they have the forms given by (ii) and (iii) respectively. Furthermore, $\eta$ has the form (\ref{eq_lem_kmn4_combinaciones_Gamma4_3}). By the normality of $\Gamma_p\rtimes\prodint{\gamma}\rtimes\prodint{\mu}$ in $\Gamma$ we have $\eta g \gamma\mu\eta^{-1}=g^{n}\gamma^{m_1}\mu^{m_2}$ for some integers $n,m_1,m_2\in\Z$. Comparing entries 12 and 13 in both sides in the previous equation and solving for $\eta_{13}$ and $k$ yields the expressions for these variables. On the other hand, $\corchetes{\mu,\eta}\in\Gamma_p$, let $\xi$ be the entry 13 of $\corchetes{\mu,\eta}$, then 
		$$\xi=\frac{q\beta\mu_{13}\parentesis{1-r^2\delta^3}}{r\delta^2\parentesis{1-q^2\beta^3}}.$$
	If $\corchetes{\mu,\eta}=\id$ then $\xi=0$, if $\corchetes{\mu,\eta}\neq\id$ the $\xi\in\Z$. This completes the proof. 
	\end{enumerate}
\end{proof}

\begin{lem}\label{lem_kmn4_combinaciones_Gamma5}
Suppose that $\Gamma_p$ is conjugate to $\Gamma_5=\prodint{g_{0,1}}$.
	\begin{enumerate}
	\item If $\rank(N_3)=1$, one generator of $N_3$ is the element
		\begin{equation}\label{eq_lem_kmn4_combinaciones_Gamma5_2}
		\gamma=\corchetes{\begin{array}{ccc}
	p & \gamma_{12} & \gamma_{13}\\
	0 & 1 & \gamma_{23}\\
	0 & 0 & 1\\
	\end{array}},
	\begin{array}{l}
	p\in\Z\setminus\set{-1,0,1}\\
	\gamma_{12},\gamma_{13}\in\C\text{, }\gamma_{23}\in\C^\ast.
	\end{array}
	\end{equation}				
	\item It's not possible to have $\rank(N_3)=2$.
	\item If $\rank(N_3)=0$ and $\rank(N_4)=1$, one generator of $N_4$ is the element
		\begin{equation}\label{eq_lem_kmn4_combinaciones_Gamma5_0}
		\gamma=\corchetes{\begin{array}{ccc}
		\alpha & \gamma_{12} & \gamma_{13}\\
		0 & p\alpha^{-2} & \gamma_{23}\\
		0 & 0 & p^{-1}\alpha\\
		\end{array}},
		\begin{array}{l}
		\alpha\in\C^{\ast},\;p\in\Z\setminus\set{0}\\
		p^2\neq \alpha^3,\\
		\gamma\text{ is loxodromic.}
		\end{array}
		\end{equation}
	\item If $\rank(N_3)=1$, $\Gamma$ cannot contain elements in its fourth layer.	
	\item If $\rank(N_3)=0$ and $\rank(N_4)=2$, one generator of $N_4$ is given by (\ref{eq_lem_kmn4_combinaciones_Gamma5_0}), and the other by
		\begin{equation}\label{eq_lem_kmn4_combinaciones_Gamma5_00}
		\mu=\corchetes{\begin{array}{ccc}
		\beta & 0 & \frac{jp}{1-p}\beta\\
		0 & q\beta^{-2} & 0\\
		0 & 0 & q^{-1}\beta\\
		\end{array}},
		\begin{array}{l}
		\beta\in\C^\ast,\;q,j\in\Z\setminus\set{0}\\
		q^2\neq\beta^3,\;\mu\text{ is loxodromic.}
		\end{array}
		\end{equation}
	\item It is not possible to have $\rank(N_4)=3$.
	\end{enumerate}
\end{lem}

\begin{proof}
Assume that $\Gamma_p=\Gamma_5$ and let $g=g_{0,n}\in\Gamma_p$, for $n\in\Z$. We verify each conclusion:
\begin{enumerate}[(i)]
	\item Let $\gamma\in N_3$, then $\gamma$ has the form given by (\ref{eq_dem_lem_kmn4_combinaciones_Gamma4_2}). As before, the normality of $\Gamma_p$ in $\Gamma$ implies that $\gamma g \gamma^{-1}\in\Gamma_p$, and this implies that $\alpha^{-3}n\in \Z$ for all $n\in\Z$. Taking $n=1$, yields $\alpha=p^{-\frac{1}{3}}$. Substituting this in the form of $\gamma$ verifies (\ref{eq_lem_kmn4_combinaciones_Gamma5_2}).
	
	\item Let us suppose that $\rank(N_2)=2$ with $N_2=\prodint{\gamma,\mu}$ where $\gamma$ has the form (\ref{eq_lem_kmn4_combinaciones_Gamma5_2}). Then $\mu$ must have the same form and then we denote
	  $$\mu=\corchetes{\begin{array}{ccc}
	q & \mu_{12} & \mu_{13}\\
	0 & 1 & \mu_{23}\\
	0 & 0 & 1\\
	\end{array}},$$
	for some $q\in\Z$. Lemma \ref{obs_los_dos_calculos} (ii) states that $\corchetes{\gamma,\mu}\in\Gamma_p\setminus\set{\id}$ or $\corchetes{\gamma,\mu}=\id$. In either case, the entry 12 of $\corchetes{\gamma,\mu}$ satisfies
	\begin{equation}\label{eq_dem_lem_kmn4_combinaciones_Gamma5_2}
	\frac{\gamma_{12}(1-q)-\mu_{12}(1-p)}{pq} = 0.	
	\end{equation}	 
This means that 
	$$\fix\parentesis{\corchetes{\begin{array}{ccc}
	q & \mu_{12} \\
	0 & 1 \\
	\end{array}}}=\fix\parentesis{\corchetes{\begin{array}{ccc}
	p & \gamma_{12} \\
	0 & 1 \\
	\end{array}}},$$
and therefore, conjugating $\Gamma$ by $A=g_{-\gamma_{12}\parentesis{1-p}^{-1},0}$ and using (\ref{eq_dem_lem_kmn4_combinaciones_Gamma5_2}), we can assume that
	$$\gamma=\corchetes{\begin{array}{ccc}
	p & 0 & \gamma_{13}\\
	0 & 1 & \gamma_{23}\\
	0 & 0 & 1\\
	\end{array}}\text{, and }\mu=\corchetes{\begin{array}{ccc}
	q & 0 & \mu_{13}\\
	0 & 1 & \mu_{23}\\
	0 & 0 & 1\\
	\end{array}},$$
for some $p,q\in\Z\setminus\set{0}$. Then
	\begin{align}
	\text{If }\corchetes{\gamma,\mu}=\id &\Rightarrow \gamma_{13}(1-q)=\mu_{13}(1-q).\label{eq_dem_lem_kmn4_combinaciones_Gamma5_3} \\
	\text{If }\corchetes{\gamma,\mu}\neq\id &\Rightarrow \frac{\gamma_{13}{1-q}-\mu_{13}(1-p)}{pq}\in\Z. \nonumber
	\end{align}
    Since $\Gamma_p\rtimes\prodint{\gamma}$ is normal in $\Gamma$, then $\mu g \gamma \mu^{-1}=g^n\gamma^k$ for some $n,k\in\Z$. Comparing entries 23 of both expressions yields $k=1$, comparing entries 13 we have
	\begin{equation}\label{eq_dem_lem_kmn4_combinaciones_Gamma5_5}
	q(1+\gamma_{13})+\mu_{13}(1-p)=1+\gamma_{13}.
	\end{equation}
    If $\corchetes{\gamma,\mu}=\id$, (\ref{eq_dem_lem_kmn4_combinaciones_Gamma5_3}) yields $q=1$, which cannot happen. If $\corchetes{\gamma,\mu}=\id$, (\ref{eq_dem_lem_kmn4_combinaciones_Gamma5_3}) implies $\parentesis{\gamma_{13}{1-q}-\mu_{13}(1-p)}\parentesis{pq}^{-1}=j$, for some $j\in\Z$. Combining this equation with (\ref{eq_dem_lem_kmn4_combinaciones_Gamma5_5}), we have $q=\parentesis{1+jp}^{-1}\in\Z$, which cannot happen. Therefore there cannot be an element in $N_3\setminus\prodint{\gamma}$.

	\item The proof is similar to the proof of (i) above.
	
	\item If $\rank(N_3)=1$, one generator of $N_3$ is the element $\gamma$ given by (\ref{eq_lem_kmn4_combinaciones_Gamma5_2}). Let us assume that there is an element $\mu\in N_4$, given by
		$$\mu=\corchetes{\begin{array}{ccc}
	\alpha & \mu_{12} & \mu_{13}\\
	0 & \beta & \mu_{23}\\
	0 & 0 & \alpha^{-1}\beta^{-1}\\
	\end{array}},\;
	\begin{array}{l}
	\alpha,\beta\in\C^{\ast},\;\alpha\beta^{-2}\neq 1\\
	\mu\text{ is loxodromic.}	
	\end{array}		
	$$
	Since $\corchetes{\gamma,\mu}\in\Gamma_p$ and comparing entries 23, it follows $\gamma_{23}\alpha^{-1}\beta^{-2}(1-\alpha\beta^2)=0$. This means that, either $\gamma_{23}=0$ or $\alpha\beta^2=1$. Both conditions contradict the hypotheses, then $\Gamma$ cannot contain elements in $N_4$.
		
	\item If $\rank(N_4)=2$, let $\gamma=\corchetes{\gamma_{ij}}$ and $\mu$ be two generators of $N_4$, then $\gamma$ has the form given by (\ref{eq_lem_kmn4_combinaciones_Gamma5_0}), we can conjugate $\Gamma$ by an adequate element of $\PSL$ and assume that $\gamma_{12}=\gamma_{13}=\gamma_{23}=0$ in (\ref{eq_lem_kmn4_combinaciones_Gamma5_0}). Since $\corchetes{\mu,\gamma}\in\Gamma_p=\Gamma_5$, $\mu$ has the form
		$$\mu=\corchetes{\begin{array}{ccc}
	\beta & 0 & \mu_{13}\\
	0 & q\beta^{-2} & 0\\
	0 & 0 & q^{-1}\beta\\
	\end{array}},\;
	\begin{array}{l}
	\beta\in\C^\ast,\;q\in\Z\setminus\set{0}\\
	q^2\neq\beta^3,\;\mu\text{ is loxodromic.}
	\end{array}
	$$
    If $\corchetes{\mu,\gamma}=\id$, then $\Gamma$ would be abelian. Therefore, $\corchetes{\mu,\gamma}\neq\id$, comparing entries 13 of $\corchetes{\mu,\gamma}$ and $g$ yields $\mu_{13}p^{-1}\beta^{-1}(1-p)=j$, for some $j\in\Z\setminus\set{0}$. Then $\mu_{13}=jp\beta(1-p)^{-1}$, this verifies (\ref{eq_lem_kmn4_combinaciones_Gamma5_00}).
\end{enumerate}
\end{proof}

\subsection{The parabolic part is an Inuoe group}



Using the notation of (5) of Theorem \ref{thm_descripcion_parte_parabolica}, it holds $r\parentesis{\text{Ino}(\mathbf{w})}=3$. This dismisses all subcases 1mn(5) and 2mn(5). In the following lemma, we dismiss the remaining subcases 310(5) and 301(5).

\begin{lem}\label{lem_kmn5_combinaciones}
Let $\Gamma\subset U_+$ be such that $\Gamma_p$ is conjugate to $\text{Ino}(\mathbf{w})$, for some $\mathbf{w}$ as in the case (5) of Theorem \ref{thm_descripcion_parte_parabolica}. 
Then $\Gamma$ cannot contain loxodromic elements.
\end{lem}

\begin{proof}
We can assume that $\Gamma_p=\text{Ino}(\mathbf{w})$. We denote the 3 generators of $\Gamma_p$ by $g_1,g_2,g_3$. Since $d\neq 0$, it follows from Proposition \ref{prop_casos_kmn4_conjugaciones} that $\text{Ino}(\mathbf{w})$ is conjugate to 
	$$
	G=\SET{g_{1,0}^k\, g_{0,1}^l\corchetes{\begin{array}{ccc}
	1 & x-\frac{ps}{q}y & sy\\
	0 & 1-\frac{ps}{q} & s\\
	0 & -\frac{ps^2}{q} & 1+\frac{ps}{q}\\
	\end{array}}^m}{k,l,m\in \Z}.$$
We consider these generators of $\Gamma_p$ instead, we still denoted them by $g_1,g_2,g_3$.

Let us suppose first that there is an element $\gamma\in N_3$. Observe that $\prodint{g_1,g_2}$ is normal in $\Gamma$ and therefore $\gamma g_1 g_2 \gamma^{-1}\in\prodint{g_1,g_2}$, Lemma \ref{lem_kmn4_combinaciones_Gamma1} yields
	$$\gamma=\corchetes{\begin{array}{ccc}
	p & \gamma_{12} & \gamma_{13}\\
	0 & 1 & \frac{q}{p}\\
	0 & 0 & 1\\
	\end{array}},\;
	\begin{array}{l}
	p\in\Z\setminus\set{-1,0,1}\\
	q\in\Z\setminus\set{0}.
	\end{array}
	$$
Since $\Gamma_p$ is normal in $\Gamma_p\rtimes\prodint{\gamma}$ then $\gamma g_3 \gamma^{-1}=g_1^k g_2^n g_3^m$ for some $k,n,m\in\Z$. Comparing entries 32 yields $m=1$, substituting in the equation resulting of comparing entries 22 yields $q=0$, which is a contradiction.
	
Now, assume that there is a loxodromic element $\gamma\in N_4$. Using the same argument as in the case of an element in the third layer we know that 
	$$\gamma=\corchetes{\begin{array}{ccc}
	ab & \gamma_{12} & \gamma_{13}\\
	0 & a & c\\
	0 & 0 & b\\
	\end{array}},\;
	\begin{array}{l}
	a,b,c\in\Z\\
	p,q\neq 0.
	\end{array}
	$$
Considering the entry 12 of $\corchetes{\gamma,g_1}\in\prodint{g_1,g_2}$ yields $1-\frac{1}{b}\in\Z$, then $b=\pm 1$. If $b=1$ then $\gamma$ is a complex homothety, which cannot happen by Proposition \ref{cor_no_hay_CH_en_no_conmutativos}. If $b=-1$ then $\lambda_{12}(\gamma)$ is a torsion element in the torsion free group $\lambda_{12}(\Gamma)$. This concludes the proof.
\end{proof}

\subsection{The parabolic part is an extended Kodaira group}

Observe that $3\leq r(\text{Kod}_1(W,x,y,z)) \leq 4$ depending on whether $r(W)=1$ or $r(W)=2$ respectively. If $\Gamma\subset U_+$ is a discrete groups with parabolic part conjugated to a non-abelian Kodaira group of type 1 $\text{Kod}_1(W,x,y,z)$ such that $r(\text{Kod}_1(W,x,y,z))=4$, then $\Gamma$ cannot contain any loxodromic elements (Theorem \ref{thm_descomposicion_caso_noconmutativo2}). Because of this, we can only consider non-abelian Kodaira group of type 1, $\text{Kod}_1(W,x,y,z)$, such that $r(W)=1$. On the other hand, non-abelian Kodaira group of type 2, $\text{Kod}_2(W,x,y,z)$, always satisfy that $r(W)=1$. 

The last paragraph implies that subcases 1mn(6) and 2mn(6) cannot happen. Therefore, we just need to describe subcases 310(6) and 301(6), but these cases cannot happen either, as we prove in the following lemma.

\begin{lem}\label{lem_kmn6_combinaciones}
If $\Gamma$ is such that $\Gamma_p$ is conjugate to a non-abelian Kodaira group of type 1 or 2 $\text{Kod}_i(W,x,y,z)$, for some $W,x,y,z$ as in Theorem \ref{thm_descripcion_parte_parabolica}(6), then $\Gamma$ cannot contain loxodromic elements.
\end{lem} 

\begin{proof}
Assume that $\Gamma_p=\text{Kod}_i(W,x,y,z)$ for $i=1,2$, with $W,x,y,z$ as in the hypothesis of the lemma. Denote by $\phi_1,\phi_2,\phi_3$ the respective generators of $\Gamma_p$ given in Theorem \ref{thm_descripcion_parte_parabolica}(6). 

Let us assume first that there is a loxodromic element $\gamma\in N_3$ in the third layer of $\Gamma$, then
	\begin{equation}\label{eq_dem_lem_kmn6_combinaciones1}
	\gamma=\corchetes{\begin{array}{ccc}
	\alpha^{-2} & \gamma_{12} & \gamma_{13}\\
	0 & \alpha & \gamma_{23}\\
	0 & 0 & \alpha\\
	\end{array}}
	\end{equation}		
for some $\alpha\in\C^\ast$ such that $\valorabs{\alpha}\neq 1$ and $\gamma_{23}\neq 0$. A direct calculation shows that
	\begin{equation}\label{eq_dem_lem_kmn6_combinaciones2}
	\Gamma_p = \SET{\corchetes{\begin{array}{ccc}
	1 & mn(x+z) & \psi_{k,m,n}\\
	0 & 1 & mnz\\
	0 & 0 & 1\\
	\end{array}}}{k,m,n\in\Z}.
	\end{equation}
The explicit expression for $\psi_{k,m,n}$ is not relevant to the proof. Using the normality of $\Gamma_p$ in $\Gamma$ and comparing the entries 23 of $\gamma \phi_1 \phi_2 \phi_3 \gamma^{-1}\in \Gamma_p$ and (\ref{eq_dem_lem_kmn6_combinaciones2}), we get that $z+1=mnz$ for some $m,n\in\Z$. Since $z\neq 0$, it follows that $1+z^{-1}\in\Z$, and then $z\in\Q$. This contradicts the conditions for $z$ in non-abelian Kodaira groups of type 1 and 2. Thus, $\Gamma$ cannot contain elements in its third layer.

Now, assume that $\gamma\in N_4$ is a loxodromic element in the fourth layer, we denote
	$$\gamma=\corchetes{\begin{array}{ccc}
	\alpha & \gamma_{12} & \gamma_{13}\\
	0 & \beta & \gamma_{23}\\
	0 & 0 & \alpha^{-1}\beta^{-1}\\
	\end{array}},\;\alpha,\beta\in\C^\ast.$$	 
Since $\Gamma_p$ is normal in $\Gamma$, and comparing entries 12 and 23 of $\gamma \phi_1 \phi_3 \gamma^{-1}\in \Gamma_p$ and (\ref{eq_dem_lem_kmn6_combinaciones2}), we get
	\begin{equation}\label{eq_dem_lem_kmn6_combinaciones4}
	\alpha\beta^{-1} = mn,\text{ and }\;\; \alpha\beta^2 = mn 
	\end{equation}
for some $m,n\in\Z$. From (\ref{eq_dem_lem_kmn6_combinaciones4}) it follows that $\beta^3=1$, and $\alpha=p\beta$ for some $p\in\Z$. Therefore,
	\begin{equation}\label{eq_dem_lem_kmn6_combinaciones6}
	\gamma=\corchetes{\begin{array}{ccc}
	p\beta & \gamma_{12} & \gamma_{13}\\
	0 & \beta & \gamma_{23}\\
	0 & 0 & p^{-1}\beta^{-2}\\
	\end{array}},\;\beta^3=1.
	\end{equation}
Since $\prodint{\phi_1}$ is a normal subgroup of $\Gamma$, it follows that $\gamma \phi_1 \gamma^{-1}\in \prodint{\phi_1}$. Since $r(W)=1$, Lemma \ref{lem_kmn4_combinaciones_Gamma5}(3) implies that 
	\begin{equation}\label{eq_dem_lem_kmn6_combinaciones7}
	\gamma=\corchetes{\begin{array}{ccc}
	\alpha & \gamma_{12} & \gamma_{13}\\
	0 & q\alpha^{-2} & \gamma_{23}\\
	0 & 0 & q^{-1}\alpha^{-1}\\
	\end{array}},\;
	\begin{array}{l}
	q\in\Z\setminus\set{0},\;\alpha\in\C^\ast\\
	\gamma\text{ is loxodromic.}
	\end{array}
	\end{equation}	
Comparing the diagonal entries of both (\ref{eq_dem_lem_kmn6_combinaciones6}) and (\ref{eq_dem_lem_kmn6_combinaciones7}), and since $\beta^3=1$ it follows that $p=\pm 1$, $q=1$, $\alpha=p$ and $\beta=1$. Then $\gamma$ is either parabolic or induces a torsion element in $\lambda_{12}(\Gamma)$, both possibilities cannot occur. This finishes the proof.
\end{proof}

\section{Proofs of the Main Theorems}\label{sec:proofs}

In this section, we prove Theorems \ref{thm:main_1}, \ref{thm:main_2}, \ref{thm:main_3} and \ref{thm:main_4}. In all proofs, $\Gamma\subset U_+$ is a non-cyclic, upper triangular discrete group containing loxodromic elements. We assume, without loss of generality, that the groups $\Gamma$, $\lambda_{12}(\Gamma)$, $\lambda_{23}(\Gamma)$ are torsion free (Proposition \ref{prop:assume_torsion_free}). 

\begin{proof}[Proof of Theorem \ref{thm:main_1}]
Suppose that $\Gamma$ is abelian. If $\Gamma_p=\set{\id}$, then $\Gamma$ contains only loxodromic elements and therefore contains only diagonal elements. Propositions 24 and 25 of \cite{toledo2021dynamics} verify (1). 

If $\Gamma_p\neq \set{\id}$, Theorem \ref{thm_descripcion_parte_parabolica} and Lemma \ref{lem:elliptic_abelian} imply that $\Gamma_p$ is conjugate to an elliptic group. Since $\Gamma$ contains loxodromic elements, Propositions 20 and 23 of \cite{toledo2021dynamics} yield (2).
\end{proof}

\begin{proof}[Proof of Theorem \ref{thm:main_2}]
If $\Gamma$ is not abelian and its parabolic part $\Gamma_p$ is conjugate to a torus group $\mathcal{T}(W)$, then Lemma \ref{lem_casos_kmn2_unico} yields the conclusion. Lemmas \ref{lem_casos_kmn2_noexisten} and Lemma \ref{lem_casos_kmn2_noexisten_2} prove that this is the only possible case.
\end{proof}

\begin{proof}[Proof of Theorem \ref{thm:main_3}]
If $\Gamma$ is not abelian and its parabolic part $\Gamma_p$ is conjugate to a dual torus group $\mathcal{T}^\ast(W)$ with $r(W)=2$, then $\Gamma_p$ is conjugate to either $\Gamma_1$, $\Gamma_2$ or $\Gamma_3$ as given by Proposition \ref{prop_casos_kmn4_conjugaciones} (i),(ii),(iii). 

If $\Gamma_p\cong\Gamma_1$, Lemma \ref{lem_kmn4_combinaciones_Gamma1} verifies conclusions (1), (2) and (3). If $\Gamma_p\cong\Gamma_2$, Lemma \ref{lem_kmn4_combinaciones_Gamma2} verifies conclusions (4) and (5). If $\Gamma_p\cong\Gamma_3$, Lemma \ref{lem_kmn4_combinaciones_Gamma3} verifies conclusions (6) and (7).

In all cases, the same lemmas prove that there are no other possible cases.
\end{proof}

\begin{proof}[Proof of Theorem \ref{thm:main_4}]
If $\Gamma$ is not abelian and its parabolic part $\Gamma_p$ is conjugate to a dual torus group $\mathcal{T}^\ast(W)$ with $r(W)=1$, then $\Gamma_p$ is conjugate to either $\Gamma_4$, $\Gamma_5$ as given by Proposition \ref{prop_casos_kmn4_conjugaciones} (iv),(v). 

If $\Gamma_p\cong\Gamma_4$, Lemma \ref{lem_kmn4_combinaciones_Gamma4} verifies conclusions (1), (2) and (3). If $\Gamma_p\cong\Gamma_5$, Lemma \ref{lem_kmn4_combinaciones_Gamma5} verifies conclusions (4), (5) and (6).

In all cases, the same lemmas prove that there are no other possible cases.
\end{proof}

Using Lemmas \ref{lem_Gamma_RLW_no_tiene_lox}, \ref{lem_kmn5_combinaciones}, \ref{lem_kmn6_combinaciones} we conclude that there are no upper triangular, non-cyclic, discrete groups $\Gamma\subset U_+$ containing loxodromic elements such that $\Gamma_p$ is conjugated to an abelian Kodaira group, an Inoue group, or an extended Kodaira group of type 1 or 2. This verifies that the four theorems \ref{thm:main_1}, \ref{thm:main_2}, \ref{thm:main_3} and \ref{thm:main_4} describe, up to conjugation, all upper triangular, non-cyclic, discrete groups containing loxodromic elements.

\section*{Acknowledgments}
The author would like to thank to the UCIM UNAM, CIMAT, FAMAT UADY, CINC UAEM and their people for their hospitality and kindness during the writing of this paper. The author particularly thanks Angel Cano for the helpful conversations.

\bibliographystyle{alpha}
\bibliography{BibliografiaCKG}

\begin{thebibliography}{BGUN18}

\bibitem[BCN11]{bcn11cuatrolineas}
W.~Barrera, A.~Cano, and J.~P. Navarrete.
\newblock Subgroups of $\text{PSL}(3,\mathbb{C})$ with four lines in general
  position.
\newblock {\em Conformal Geometry and Dynamics}, 15:160--176, 2011.

\bibitem[BCN14]{bcn14unalinea}
W.~Barrera, A.~Cano, and J.~P. Navarrete.
\newblock One line complex {Kleinian} groups.
\newblock {\em Pacific Journal of Mathematics}, 272(2):275--303, 2014.

\bibitem[BCN16]{bcn16}
W.~Barrera, A.~Cano, and J.~P. Navarrete.
\newblock On the number of lines in the limit set for discrete subgroups of
  $\text{PSL}(3,\mathbb{\C})$.
\newblock {\em Pacific Journal of Mathematics}, 281(1):17--49, 2016.

\bibitem[BCNS22]{barrera2022}
Waldemar Barrera, Angel Cano, Juan~Pablo Navarrete, and José Seade.
\newblock Discrete parabolic groups in {PSL(3,C)}.
\newblock {\em Linear Algebra and its Applications}, 653:430--500, 2022.

\bibitem[BGUN18]{bgn18}
W.~Barrera, A.~Gonzalez-Urquiza, and J.~P. Navarrete.
\newblock Duality of the {Kulkarni} limit set for subgroups of
  $\text{PSL}(3,\mathbb{\C})$.
\newblock {\em Bulletin of the Brazilian Mathematical Society}, 49:261--277,
  2018.

\bibitem[CNS13]{ckg_libro}
A.~Cano, J.~P. Navarrete, and J.~Seade.
\newblock {\em Complex {Kleinian} Groups}.
\newblock Progress in Mathematics 303. Springer Basel, 2013.

\bibitem[CS14]{cs2014}
A.~Cano and J.~Seade.
\newblock On discrete groups of automorphisms of $\mathbb{P}_{\mathbb{c}}^2$.
\newblock {\em Geometriae Dedicata, Springer}, 168:9--60, 2014.

\bibitem[Kod64]{kodaira1964structure}
Kunihiko Kodaira.
\newblock On the structure of compact complex analytic surfaces, i.
\newblock {\em American Journal of Mathematics}, 86(4):751--798, 1964.

\bibitem[Mas87]{maskit}
B.~Maskit.
\newblock {\em Kleinian Groups}.
\newblock Grundlehren der matematischen Wissenschaften 287. Springer Verlag,
  1987.

\bibitem[MT98]{taniguchi}
K.~Matsuzaki and M.~Taniguchi.
\newblock {\em Hyperbolic Manifolds and {Kleinian} Groups}.
\newblock Clarendon Press, 1998.

\bibitem[SV01]{seade2001actions}
Jos{\'e} Seade and Alberto Verjovsky.
\newblock Actions of discrete groups on complex projective spaces.
\newblock {\em Contemporary Mathematics}, 269:155--178, 2001.

\bibitem[TA19]{mitesis}
G.~M. Toledo-Acosta.
\newblock Three generalizations regarding limit sets for complex kleinian
  groups.
\newblock {\em Centro de Investigaci\'on en Matem\'aticas, {PHD} Thesis. arXiv
  preprint arXiv:1912.02369}, 2019.

\bibitem[TA21]{toledo2021dynamics}
Mauricio Toledo-Acosta.
\newblock The dynamics of solvable subgroups of $\text{PSL}(3,\mathbb{C})$.
\newblock {\em Bulletin of the Brazilian Mathematical Society, New Series},
  pages 1--45, 2021.

\bibitem[Weh73]{wehrfritz2012infinite}
Bertram Wehrfritz.
\newblock {\em Infinite linear groups: an account of the group-theoretic
  properties of infinite groups of matrices}.
\newblock Springer-Verlag Berlin Heidelberg, 1973.

\end{thebibliography}

\end{document}